\documentclass[oneside,english]{amsart}
\usepackage[T1]{fontenc}
\usepackage[latin9]{inputenc}
\usepackage{verbatim}
\usepackage{amsthm}
\usepackage{amsbsy}
\usepackage{amstext}
\usepackage{amssymb}
\usepackage{bbm, dsfont}
\usepackage{hyperref}
\usepackage{xcolor}

\makeatletter
\numberwithin{equation}{section}
\numberwithin{figure}{section}
\theoremstyle{plain}
\newtheorem{thm}{\protect\theoremname}
  \theoremstyle{plain}
  \newtheorem{lem}[thm]{\protect\lemmaname}
  \newtheorem{cor}[thm]{\protect\corollaryname}
  \theoremstyle{remark}
  \newtheorem{rem}[thm]{\protect\remarkname}
  \theoremstyle{plain}
  \newtheorem{prop}[thm]{\protect\propositionname}
  \theoremstyle{definition}
  \newtheorem{defn}[thm]{\protect\definitionname}
\numberwithin{thm}{section}
\makeatother

\usepackage{babel}
  \providecommand{\definitionname}{Definition}
  \providecommand{\lemmaname}{Lemma}
  \providecommand{\propositionname}{Proposition}
  \providecommand{\corollaryname}{Corollary}
  \providecommand{\remarkname}{Remark}
\providecommand{\theoremname}{Theorem}

\newcommand{\R}{\mathbb{R}}
\newcommand{\C}{\mathbb{C}}

\newcommand{\Z}{\mathbb{Z}}
\newcommand{\N}{\mathbb{N}}
\newcommand{\bbP}{\mathbb{P}}
\newcommand{\Pb}[1]{\mathbb{P}\left[#1\right]}
\newcommand{\E}[1]{\mathbb{E}\left[#1\right]}
\newcommand{\PT}[1]{\mathbb{P}_{n}\left[#1\right]}

\newcommand{\PTa}[1]{\mathbb{P}_{n,\alpha}\left[#1\right]}
\newcommand{\ETa}[1]{\mathbb{E}_{n,\alpha}\left[#1\right]}
\newcommand{\PTaS}{\mathbb{P}_{n,\alpha}}

\newcommand{\caO}{{\mathcal O}}
\newcommand{\bsq}{{\boldsymbol q}}

\begin{document}

\title[Random permutations without macroscopic cycles]{Precise asymptotics of longest cycles in random permutations without macroscopic cycles}

\author{Volker Betz \and Julian M\"uhlbauer \and Helge Sch\"afer \and Dirk Zeindler}

\subjclass[2010]{60F17, 60F05, 60C05}

\begin{abstract}
We consider Ewens random permutations of length $n$
conditioned to have  no cycle longer than $n^\beta$ with 
$0<\beta<1$ and to study the asymptotic behaviour as $n\to\infty$. 
We obtain very precise information on the joint distribution of 
the lengths of the longest cycles; in particular we prove a 
functional limit theorem where the 
cumulative number of long cycles converges to a Poisson 
process in the suitable scaling. Furthermore, we prove convergence 
of the total variation distance between joint cycle counts 
and suitable independent Poisson random variables up to a 
significantly larger maximal cycle length than previously known. 
Finally, we remove a superfluous assumption from a central limit 
theorem for the total number of cycles proved in an earlier paper.

%
%
%
%
%
\end{abstract}

\maketitle

\section{Introduction}

In uniform random permutations, long cycles occupy almost all the available space. Indeed, it is a standard 
textbook exercise to show that in a permutation of length $n$, the probability to find an index $i$ in a cycle 
of length $k$ is equal to $1/n$, which in turn means that cycles of a length below volume order play no role 
asymptotically as $n \to \infty$. Of course, much more is known about uniform (and Ewens) random permutations, 
including the precise distribution of long and short cycles. We refer to 
\cite{ABT02} and the references therein. 
	
It is interesting to see how the behaviour of random permutations changes when the uniform measure is changed in 
a way that favours short cycles. Various such models have been studied in recent years. Many of them are 
motivated by the model of spatial random permutations \cite{BeUe09}, which by its close connections to Bose-Einstein condensation \cite{Ue06} has a significant physical relevance. In this model, a spatial 
structure is superimposed on the permutations, and the importance of that spatial structure is measured by an 
order parameter which physically is the temperature. It is conjectured that this order parameter mediates a 
phase transition between a regime of only short cycles and a regime of coexistence of long and short cycles. 
Despite some successes in the explicitly solvable annealed case without interaction between different cycles \cite{BeUe10}, 
and significant recent progress (using the method of reflection positivity) in a closely 
related model with such interaction \cite{LeTa19, Ta19}, many of the most relevant questions 
in spatial random permutations remain to be answered. 

A somewhat more direct and in general easier to analyse way to suppress long cycles is to introduce
cycle weights or hard constraints on cycle numbers. Cycle weights  appear in an (uncontrolled) 
approximation of the interacting Bose gas by a variant of the free one \cite{BeUe10b}, but have also been
studied intensively in their own right, both in cases 
where the cycle weights do not depend on the system size $n$ 
\cite{BeUeVe11, ErUe11}, and in cases where they do 
\cite{BoZe14, ElPe19}. In the latter case, it has been shown in the 
cited papers that one recovers the model treated in \cite{BeUe10} by 
a suitable choice of cycle weights, and the methods of analytic 
combinatorics used in \cite{BoZe14, ElPe19} yield very precise
information about the asymptotic cycle distribution in various regimes.

The present paper deals with the other option of constraining 
permutations, namely to completely disallow certain cycle lengths. 
Again, a distinction has to be made between cases where the set of 
disallowed cycle lengths is independent of the permutation length 
$n$, and those where it depends on $n$. In the first case, a 
significant amount of information has been obtained in the works of 
Yakymiv (see e.g. \cite{Ya09a,Ya10a}); our interest lies 
in the second case. Using precise asymptotic results by 
Manstavi{\v{c}}ius and Petuchovas \cite{MaPe16}, in \cite{BeSc17, BeScZe17} we 
investigated the case where a permutation of 
length $n$ is prevented from having any cycles above a threshold 
$\alpha(n)$ that grows strictly slower than volume order. 
While the results in these papers were 
reasonably detailed, some interesting questions and fine details 
have been left out. 

It is the purpose of the present paper to settle a significant
portion of them. We will describe our results in detail in the
next section. Here, we only briefly sketch what is new. 

One difference to \cite{BeSc17} is that 
we generalise the base model we constrain, from uniform 
random permutations to the model of Ewens permutations. The latter originally 
appeared in population genetics, see \cite{Ew72}, but has now become a 
rather standard model of random permutations. It shares many 
features and techniques with uniform permutations, and classical 
results about uniform and Ewens random permutations include
convergence of joint 
cycle counts towards independent Poisson random variables in total 
variation distance \cite{ArTa92c}, the convergence of the renormalized cycle structure towards a 
Poisson-Dirichlet distribution {\cite{Ki77, ShVe77}},  and a central limit theorem for 
cumulative cycle counts \cite{DePi85}.

In the context of the methods we use, the difference between the 
Ewens measure and uniform random permutations is not large, see \cite{Sc18} for details. What should be considered the main
contribution of the present paper compared to \cite{BeSc17,BeScZe17} are the following three items: 
firstly,  
we obtain much more precise asymptotics for the distribution of the longest cycles in various regimes 
(Propositions \ref{prop:LongestDiv} and \ref{prop:LongestConv}, and Theorem \ref{thm:Longest0Poissonprocess}); secondly, we extend the validity of the 
joint Poisson approximation (in variation distance) to the whole
regime of cycles of length $(o(\alpha(n))$ (Theorem \ref{thm:main_dtv}).
Finally, we remove a spurious additional assumption for the 
central limit theorem for cycle numbers that was present in 
\cite{BeScZe17}, see Theorem \ref{thm:Haupt}. 

The paper is organised as follows: in Section \ref{sec:results},
we introduce the model, give our results and compare them to 
previously existing ones. In Section \ref{sec:proofs}, we prove 
those results.

\newpage

\section{Model and Results} \label{sec:results}

\subsection{The symmetric group and the Ewens measure}
\label{sec:sym_group}

For $n\in\N$, let $S_{n}$ be the group of all permutations of the set $\{1,\ldots,n\}$.
For $\sigma\in S_n$ and $m\in\N$, we denote by $C_m(\sigma)$ the number of cycles of length $m$ 
in the cycle decomposition of $\sigma$ into disjoint cycle.
Note that we typically write $C_m$ instead of $C_m(\sigma)$.
Let $n \mapsto \alpha(n)$ satisfy the condition 
\begin{equation}
\label{eq:condition on alpha}
n^{a_1} \leq \alpha(n) \leq n^{a_2}
\end{equation}
with $a_{1},a_{2}\in\left(0,1\right)$.
We denote by $S_{n,\alpha}$ the subset of $S_{n}$ of all permutations $\sigma$ for which all cycles 
in the cycle decomposition of $\sigma$ have length at most $\alpha(n)$. 
In other words, $\sigma\in S_{n,\alpha}$ if and only if $C_m(\sigma) =0$ for $m > \alpha(n)$.
%
%
%
%
%
%
%
%
%
%
%
 For $\vartheta >0$, the Ewens measure on $S_n$ with parameter $\vartheta$ is defined as 
 \begin{align}
  \PT{\sigma}
  :=
  \frac{\prod_{m=1}^n\vartheta^{C_m(\sigma)}}{\vartheta(\vartheta+1)\cdots(\vartheta +n-1)}.
  \label{eq:def_Ewens_measure}
 \end{align}
Note that the case $\vartheta =1$ corresponds to the uniform measure.
Further, let $\mathbb{P}_{n,\alpha}$ denote the measure on $S_{n,\alpha}$ obtained by conditioning $\mathbb{P}_{n}$ on $S_{n,\alpha}$, i.e. 
\begin{align}
 \PTa{A}:= \PT{A|S_{n,\alpha}} 
\qquad
\text{ for all }
A\subset S_{n,\alpha}.
\end{align}
Inserting the definition $\mathbb{P}_{n}$, we obtain for $\sigma\in S_{n,\alpha}$ that
 \begin{align}
  \PTa{\sigma}
  =
  \frac{\prod_{m=1}^n\vartheta^{C_m(\sigma)}}{Z_{n,\alpha}\, n!}
  \quad 
  \text{ with }
  \quad 
  Z_{n,\alpha} = \frac{1}{n!}\sum_{\sigma\in S_{n,\alpha}} \prod_{m=1}^n\vartheta^{C_m(\sigma)}.
  \label{eq:def_Ewens_measure_alpha}
 \end{align}
Also, we write $\mathbb{E}_{n}$ for the expectation with respect to $\mathbb{P}_{n}$
and $\mathbb{E}_{n,\alpha}$ for the expectation with respect to $\mathbb{P}_{n,\alpha}$.
%
%
%
%
%
\subsection{Notation}
If two sequences $(a_n)$ and $(b_n)$ are asymptotically 
equivalent, i.e.\ if $\lim_{n\to\infty} a_n/b_n = 1$, we write 
$a_n \sim b_n$. 
Further, we write $a_n\approx b_n$ when there exist constants $c_1,c_2>0$ such that
\begin{align}
 c_1 b_n \leq a_n \leq c_2 b_n
\end{align}
for large $n$. We also use the usual $\caO$ and $o$ notation,
i.e. $f(n) = \caO(g(n))$ means that there exists some constant 
$c > 0$ so that $|f(n)| \leq c |{g(n)}|$ for large $n$,
while $f(n) = o(g(n))$ means that for all $c>0$ there exists 
$n_c \in \N$ so that the inequality $|f(n)| \leq c |g(n)|$
holds for all $n > n_c$. 
We further say that
$$
f_n(t)
=
\mathcal{O}\left(g_n(t)\right)\text{ uniformly in }t\in T_n \text{ as }n\to\infty
$$
if there are constants $c,N>0$ such that
$
\sup_{t\in T_n} |f_n(t)|\leq c |g_n(t)|
$
for all $n\geq N$. 

\subsection{Expected cycle counts} 
Here we recall some of the results from \cite{BeScZe17} and 
\cite{Sc18} that are crucial for the following.

Let $x_{n,\alpha}$ be the unique positive solution of the equation 
\begin{equation}
n=\vartheta \sum_{j=1}^{\alpha(n)}x_{n,\alpha}^{j},
\label{eq:StaSad}
\end{equation}
and
\begin{align}
\mu_m\left(n\right)
:=
\vartheta\frac{x_{n,\alpha}^{m}}{m}.
\label{eq:def_mu_n}
\end{align}
For the case where $m$ is replaced by an integer-valued 
sequence $(m(n))_{n \in 
\mathbb{N}}$, we simplify notation and write $\mu_{m(n)}$ instead of 
$\mu_{m(n)}(n)$. For any such sequence that satisfies 
$m\left(n\right)\leq\alpha\left(n\right)$, we have 
\begin{equation} \label{eqn:exp asympt}
\ETa{C_{m\left(n\right)}}
\sim
\mu_{m\left(n\right)}
\qquad \text{ as } n\to\infty.
\end{equation}
%
This was proven for $\vartheta =1$ in \cite[Proposition 2.1]
{BeScZe17}, and for $\vartheta \neq 1$ in \cite{Sc18} along the 
same lines. In view of \eqref{eqn:exp asympt} it is clear that we 
are interested in information about the asymptotics of solutions 
to equations like \eqref{eq:StaSad}. 
The following result provides it: 

\begin{lem}
\label{lem:saddle_point_with_c}
Let $0<c_1<c_2<\infty$ be fixed, but arbitrary real numbers. 
For $c \in[c_1,c_2]$, let $x_{n,\alpha}(c)$  be the solution of
\begin{align}
cn = \vartheta\sum_{j=1}^{\alpha(n)} \big( x_{n,\alpha}(c) \big)^j.
\label{eq:def_xn(c)}
\end{align}
We then have uniformly in $c \in[c_1,c_2]$ as $n\to\infty$ 
\begin{equation}
\alpha\left(n\right)\log\left(x_{n,\alpha}(c)\right)
=
\log\left(\frac{cn}{\vartheta\alpha\left(n\right)}\log\left(\frac{cn}{\vartheta\alpha\left(n\right)}\right)\right)
+
\caO\left(\frac{\log\left(\log\left(n\right)\right)}{\log\left(n\right)}\right).
\label{eq:StaSadAs}
\end{equation}
In particular,
$$x_{n,\alpha}(c)\geq1, \ 
\lim_{n\rightarrow\infty}x_{n,\alpha}(c)=1
\ \text{ and } \
\big(x_{n,\alpha}(c)\big)^{\alpha\left(n\right)}
\sim
\frac{cn}{\vartheta\alpha\left(n\right)}\log\left(\frac{cn}{\vartheta\alpha\left(n\right)}\right)$$
for large $n$. Furthermore,
\begin{equation}\label{eq:lambda2alt}
\sum_{j=1}^{\alpha(n)} j\big(x_{n,\alpha}(c)\big)^j \sim \frac{cn}{\vartheta}\alpha(n).
\end{equation}
\end{lem}
%
Lemma~\ref{lem:saddle_point_with_c} is  a special case of \cite[Lemma~9]{MaPe16} and follows immediately by inserting our assumptions in \cite[Lemma~9]{MaPe16}.
We thus omit the proof.

\subsection{Asymptotics of longest cycles}

The first set of results that we present deals with the asymptotic
(joint) distribution of the longest cycles under the measure 
$\PTaS$.
Let $\ell_{k}=\ell_{k}\left(\sigma\right)$ denote 
the length of the $k$-th 
longest cycle of the permutation $\sigma$. 
We already know that for fixed $K \in \mathbb{N}$,
under the probability measures $\PTaS$,
we have as $n\to\infty$
\begin{align}
 \frac{1}{\alpha\left(n\right)}\left(\ell_{1},\ell_{2},\dots,\ell_{K}\right)\stackrel{d}{\longrightarrow}\left(1,1,\dots,1\right),
\end{align}
where $\stackrel{d}{\longrightarrow}$ denotes convergence in
distribution (see equation (2.14) in \cite{BeScZe17} or  
\cite{Sc18}). We will significantly improve on this information. 

It turns out that the behaviour of the longest cycles depends on 
the expected length given in \eqref{eqn:exp asympt}. 
In other words, we have to look at the behaviour of $\mu_{\alpha\left(n\right)}$ 
in the three regimes 
$$\mu_{\alpha\left(n\right)}\to\infty,\ 
\mu_{\alpha\left(n\right)} \to \mu  \text{ with } \mu>0
\ \text{, and } \ 
\mu_{\alpha\left(n\right)}\to 0.$$
A discussion about which regime happens when in case of \
$\alpha(n) = n^\beta$ can be found in Section 2.2 of 
\cite{BeScZe17}. 

We start with the simplest case $\mu_{\alpha\left(n\right)}\to\infty$. 
This case only occurs if $\alpha\left(n\right)=o((n\log n)^{\frac{1}{2}})$, see Proposition~\ref{prop:asymptotic_mu}
below. In this case, the distribution of the random vector 
$(\ell_1, \ldots, \ell_K)$ becomes degenerate: 

\begin{prop}
\label{prop:LongestDiv}
Suppose that $\mu_{\alpha(n)}\to\infty$.
Then, for each $K\in\mathbb{N}$, we have
\[
\lim_{n\to\infty}\PTa{\left(\ell_{1},\ell_{2},\dots,\ell_{K}\right) \neq \big(\alpha\left(n\right),\alpha(n),\dots,\alpha(n)\big)}=0.
\]
\end{prop}
A similar proposition was proven in \cite[Theorem 2.8]{BeScZe17} and \cite{Sc18} 
under the additional assumption that $\alpha(n)\geq n^{\frac{1}{7}+\delta}$ for $\delta>0$. 
The reason why we can omit this assumption here is our improved central limit theorem, 
Theorem \ref{thm:Haupt}. We give the  proof of Proposition 
\ref{prop:LongestDiv} in Section 
\ref{sect:proof_longest_1}.

Next, we now look at the case $\mu_{m(n)} \to \mu$ with $\mu>0$. 
We find
\begin{prop}
\label{prop:LongestConv}
Suppose that $\mu_{\alpha\left(n\right)} \to \mu$ with $\mu>0$ as $n\to\infty$. 
We then have for all $d\in\mathbb{N}_{0}$ and all $k\in\N$ that
\begin{align}
 \PTa{\ell_{k}=\alpha\left(n\right)-d}
 \xrightarrow{n\to\infty}
 \frac{1}{\Gamma\left(k\right)}\int_{d\mu}^{\left(d+1\right)\mu}v^{k-1}\mathrm{e}^{-v}\mathrm{d}v.
\end{align}
In other words, $\alpha\left(n\right)-\ell_{k}$
converges in distribution to $\left\lfloor \mu^{-1}X\right\rfloor $,
where $X$ is a gamma-distributed random variable with parameters
$k$ and $1$ and $\lfloor x\rfloor = \max\{n\in\Z;\, n\leq x \}$.
\end{prop}
The proof of this proposition is given in Section \ref{sect:proof_longest_2}. Moreover, the proof allows for deriving the joint distribution of the longest cycles, but the notation of results in this case is cumbersome.

Finally, we have the case where the expected number of cycles vanishes. Here we obtain the most interesting results, namely a functional convergence of the cumulative
numbers of long cycles to a Poisson process, on the correct scale. 
By considering the jump times of this Poisson process, we establish limit
theorems for $\ell_{k}$.
Let us start with a small observation.
\begin{prop}
\label{prop:asymptotic_mu}
We have, as $n\to\infty$,
\begin{align}
  \mu_{\alpha(n)} 
  \approx
  \frac{n\log n}{(\alpha\left(n\right))^2}.
  \label{eq:LongestConv0MuN}
\end{align}
\end{prop}
\begin{proof}
 Inserting the definition of $\mu_{\alpha(n)}$, see \eqref{eq:def_mu_n}, and using Lemma~\ref{lem:saddle_point_with_c}, we obtain
 \begin{align}
  \mu_{\alpha(n)} 
  =
  \vartheta \frac{x_{n,\alpha}^{\alpha(n)}}{\alpha(n)}
  \sim
  \frac{n}{(\alpha\left(n\right))^2}\log\left(\frac{n}{\vartheta\alpha\left(n\right)}\right)
  \approx
  \frac{n\log n}{(\alpha\left(n\right))^2}.
 \end{align}
This completes proof of this proposition.
\end{proof}
This proposition immediately implies that $\mu_{\alpha(n)} \to 0$ if and only if  $\frac{n\log n}{(\alpha\left(n\right))^2}\to 0$ as $n\to\infty$.
We now define 
\begin{align}
 d_{t}\left(n\right)
:=
\max\left\{ \alpha\left(n\right)-\left\lfloor \frac{t}{\mu_{\alpha\left(n\right)}}\right\rfloor ,0\right\}.
\label{eq:def_dt}
\end{align}
Note that $d_{t}(n) = \alpha(n)(1 + o(1))$ and $\left\lfloor \frac{t}{\mu_{\alpha\left(n\right)}}\right\rfloor \to\infty$ if $\mu_{\alpha(n)}\to 0$ for fixed $t$.
We now have
\begin{thm}
\label{thm:Longest0Poissonprocess}
Suppose that $\mu_{\alpha(n)}\to 0$ and define for $t\geq 0$ 
\[
P_{t}
:=
\sum_{j=d_{t}\left(n\right)+1}^{\alpha\left(n\right)}C_{j}.
\]
Then the stochastic process $\left\{P_{t}, t\geq0 \right\}$ converges under $\PTaS$ as $n\to\infty$
weakly in $\mathcal{D}\left[0,\infty\right)$ to a Poisson process
with parameter $1$, where  $\mathcal{D}\left[0,\infty\right)$ denotes the space of c\`{a}dl\`{a}g-functions.
%
\end{thm}
This theorem is proved in Section \ref{sect:proof_longest_3}.
It immediately implies the following corollary. 
\begin{cor}
\label{cor:Longest0}
Let $K\in\mathbb{N}$ be given, $\alpha(n)$ be as in \eqref{eq:condition on alpha} and suppose that $\mu_{\alpha\left(n\right)}\to 0$.
We have convergence in distribution of
\[
\mu_{\alpha\left(n\right)}\cdot\left(\alpha\left(n\right)-\ell_{1},\ell_{2}-\ell_{1},\dots,\ell_{K}-\ell_{K-1}\right)
\]
under $\PTaS$ to independent
exponentially distributed random variables with parameters $1$. In
particular, $\mu_{\alpha\left(n\right)}\left(\alpha\left(n\right)-\ell_{k}\right)$
converges in distribution to a gamma-distributed random variable with
parameters $k$ and $1$.
\end{cor}

\begin{proof}
The claim is a consequence of the convergence established in the proof
of Theorem \ref{thm:Longest0Poissonprocess} since the limit distribution
is the distribution of the jump times of the Poisson process (see, e.g. \cite[p.5]{Li10}).
\end{proof}

\subsection{Total variation distance}

Here we study the joint behaviour of the cycle counts $C_m$
in the region $m=o(\alpha(n))$.
Recall that the total variation distance of two probability measures $\bbP$ and 
$\widetilde \bbP $ on a discrete probability space $\Omega$ is given by 
$\| \bbP - \widetilde \bbP \|_{\rm TV} = \sum_{\omega \in \Omega} (\bbP(\omega) - \widetilde \bbP(\omega))_+$. 
\begin{thm}[\protect{\cite[Theorem 2.2]{BeScZe17}}] 
\label{thm:main_thm2_old}
	Let $b = (b(n))_n$ be a sequence of integers
    with $b(n) = o \big( \alpha(n) (\log n)^{-1}\big)$. Let 
	$\bbP_{n,b(n),\alpha}$ be the distribution of $(C_1, \ldots, C_{b(n)})$
	under the uniform measure on $S_{n,\alpha}$, and let $\widetilde \bbP_{b(n)}$ 
	be the distribution
	of independent Poisson-distributed random variables 
	$(Z_{1}, \ldots Z_{b(n)})$ with 
	$\widetilde{\mathbb{E}}_{b(n)}(Z_{j}) = \frac{1}{j}$ 
	for all $j \leq b(n)$. Then there exists $c<\infty$ so that for all 
	$n \in \N$, we have  
	\[
	\| \bbP_{n,b(n),\alpha} - \tilde \bbP_{b(n)} \|_{\rm TV} \leq c 
	\left( \frac{\alpha(n)}{n} + b(n) \frac{\log n}{\alpha(n)}
	\right).
	\]
\end{thm}
In the special case $ \alpha(n) \geq \sqrt{n\log(n)}$, Judkovich \cite{Ju19} has computed the above total variation distance using Steins method and obtained a slightly better upper bound.

On the full symmetric group $S_n$, a similar result as Theorem~\ref{thm:main_thm2_old} holds with $b(n)=o(n)$, see \cite{ArTa92c}.
A natural question at this point is thus if one can replace $b(n) $ in Theorem~\ref{thm:main_thm2_old} by  $b(n) = o \big( \alpha(n)\big)$.
Recall, we have seen in equation \eqref{eqn:exp asympt} that 
\[
\ETa{C_{m\left(n\right)}}
\sim
\vartheta  \frac{x_{n,\alpha}^m}{m}
\qquad \text{ as } n\to\infty.
\]
Using Lemma~\ref{lem:saddle_point_with_c}, we immediately see that $ \ETa{C_{m\left(n\right)}} \sim \E{Z_m}$ 
if and only if $m= o \big( \alpha(n) (\log n)^{-1}\big)$. 
Thus $b(n) = o \big( \alpha(n) (\log n)^{-1} \big)$ is the most one can expect in Theorem~\ref{thm:main_thm2_old}. 
To overcome the problem with the expectations, we replace the random variables $Z_j$ with fixed expectation by
random variables $Y_j^{(n)}$ with an expectation depending on $n$ so that
\begin{align}
  \ETa{C_{m\left(n\right)}} \sim \E{Y_m^{(n)}} \ \text{ for all } m= o \big( \alpha(n) \big).
\end{align}
However, to simplify the notation, we write $Y_{j}$ instead $Y_{j}^{(n)}$.
We now have
\begin{thm}
\label{thm:main_dtv}
	Let $b = (b(n))_n$ be a sequence of integers with $b(n) = o \big( \alpha(n) \big)$. 
    Let $\bbP_{n,\vartheta, b(n),\alpha}$ be the distribution of $(C_1, \ldots, C_{b(n)})$
	under $\PTaS$ on $S_{n,\alpha}$. 
    Further, let $\widehat\bbP_{b(n)}$ be the distribution
	of independent Poisson-distributed random variables 
	$(Y_{1}, \ldots, Y_{b(n)})$ with 
	$\E{Y_{j}} = \mu_{j}(n)$ for all $j \leq b(n)$
	and $\mu_{j}(n)$ as in \eqref{eq:def_mu_n}.
	Then
        \begin{align}
        \| \bbP_{n,\vartheta,b(n),\alpha} - \widehat \bbP_{b(n)} \|_{\rm TV} 
        = \caO\left(n^{\epsilon} \left(\frac{\alpha(n)}{n}\right)^{\frac{5}{12}}\right),
        \label{eq:thm:main_dtv1}
        \end{align} 
	where $\epsilon>0$ is arbitrary.
	Further, if $b(n) = o \big( \alpha(n) (\log n)^{-1}\big)$ then 
        \begin{align}
        \| \bbP_{n,\vartheta,b(n),\alpha} - \widehat \bbP_{b(n)} \|_{\rm TV} 
        = 
        \caO\left( \frac{\alpha(n)}{n} + \frac{b(n)\log n}{n^{\frac{5}{12}}\alpha^{\frac{7}{12}}}\right).
        \label{eq:thm:main_dtv2}
        \end{align}
\end{thm}
The proof of this theorem is given in Section \ref{sect:proof_dtv}.

\subsection{Central Limit Theorem for Cycle Numbers}
For the proof of Proposition~\ref{prop:LongestDiv}, we require a central limit theorem for the cycle counts in the case $\E{C_m}\to\infty$.
The main result of this section is to establish this theorem. 
Explicitly, we prove the following.
\begin{thm}
\label{thm:Haupt}
Let $m_{k}:\mathbb{N\rightarrow\mathbb{N}}$ for
$1\leq k\leq K$ such that $m_{k}\left(n\right)\leq\alpha\left(n\right)$
and $m_{k_{1}}\left(n\right)\neq m_{k_{2}}\left(n\right)$ if $k_{1}\neq k_{2}$
for large $n$. Suppose that
\[
\mu_{m_{k}\left(n\right)}\left(n\right)\rightarrow\infty
\]
for all $k$. We then have as $n\to\infty$
\[
\left(\frac{C_{m_{1}\left(n\right)}-\mu_{m_{1}\left(n\right)}\left(n\right)}{\sqrt{\mu_{m_{1}\left(n\right)}\left(n\right)}},\dots,\frac{C_{m_{K}\left(n\right)}-\mu_{m_{K}\left(n\right)}\left(n\right)}{\sqrt{\mu_{m_{K}\left(n\right)}\left(n\right)}}\right)\xrightarrow{d}\left(N_{1},\dots,N_{K}\right),
\]
with $N_{1},\ldots, N_K$ independent standard normal distributed random variables.
%
\end{thm}
This theorem was proven in \cite{BeScZe17} under the additional assumption
\begin{equation}
n^{-\frac{5}{12}}\alpha\left(n\right)^{-\frac{7}{12}}\frac{x_{n,\vartheta}^{m_{k}\left(n\right)}}{\sqrt{\mu_{m_{k}\left(n\right)}\left(n\right)}}\rightarrow0.
\label{eq:CCcltAss}
\end{equation}
In Section \ref{sect:proof_clt} we present a proof that does not 
require the addidional assumption. 

\section{Proofs}
\label{sec:proofs}
\subsection{Generating functions and the saddle point method}
\label{sec:Generating}
Generating functions and their connection with analytic combinatorics form the backbone of the proofs in this paper. 
More precisely, we will determine formal generating functions for all relevant moment-generating functions and then 
use the saddle-point method to determine the asymptotic behaviour of these moment-generating functions as $n\to\infty$.

Let $\left(a_{n}\right)_{n\in\mathbb{N}}$ be a sequence of complex numbers. Then its ordinary generating function is defined as the formal power series
\[
f\left(z\right):=\sum_{n=0}^{\infty}a_{n}z^{n}.
\]
The sequence may be recovered by formally extracting the coefficients
\[
\left[z^n\right]f\left(z\right):=a_{n}
\]
for any $n$. The first step is now to consider a special case of P{\'o}lya's Enumeration Theorem, see \cite[\S 16, p.\:17]{Po37}, 
which connects permutations with a specific generating function.
\begin{lem}
\label{lem:polya}
Let $(q_j)_{j\in\N}$ be a sequence of complex numbers. 
We then have  the following identity between formal power series in $z$,
\begin{equation}
\label{eq:symm_fkt}
\exp\left(\sum_{j=1}^{\infty}\frac{q_j z^j}{j}\right)
=\sum_{k=0}^\infty\frac{z^k}{k!}\sum_{\sigma\in S_k}\prod_{j=1}^{k}
q_{j}^{C_j},
\end{equation}
where $C_j=C_j(\sigma)$ are the cycle counts. If either of the
series in \eqref{eq:symm_fkt} is absolutely convergent, then so is
the other one.
\end{lem}
Extracting the $n$th coefficient yields
\begin{equation}
\label{eq:relation to perms}	
\left[z^n\right]\exp\left(\sum_{j=1}^{\infty}\frac{q_j z^j}{j}\right) 
= 
\frac{1}{n!}\sum_{\sigma\in S_n}\prod_{j=1}^{n}q_{j}^{C_j}.
\end{equation}
%
With this formulation, the parameters $(q_j)$ can depend on the system size $n$. 
For instance, setting $q_j =\vartheta\,\mathbbm{1}_{\left\{j\leq \alpha(n)\right\}}$,
we obtain 
\begin{align}
\label{eq:cNorm}
Z_{n,\alpha} 
= 
\left[z^n\right]\exp\left(\vartheta\sum_{j=1}^{\alpha(n)}\frac{z^j}{j}
\right)
\end{align}
with $Z_{n,\alpha}$ as in \eqref{eq:def_Ewens_measure_alpha}.
Similarly, we can get an expression for the moment generating function of $C_{m(n)}$, where $(m(n))_{n\in\N}$ is an integer sequence with $m(n)\leq \alpha(n)$.
Indeed, setting $q_{m(n)} =\vartheta e^{s}$ and $q_j =\vartheta\,\mathbbm{1}_{\left\{j\leq \alpha(n)\right\}}$ for $j\neq m(n)$, we get
\begin{align}
\label{eq:moment_Cm}
\ETa{e^{s C_{m(n)}}}
= 
\frac{1}{Z_{n,\alpha} }
\left[z^n\right]\exp\left(\vartheta (e^s-1)\frac{z^{m(n)}}{m(n)}\right) \exp\left(\vartheta\sum_{j=1}^{\alpha(n)}\frac{z^j}{j}\right).
\end{align}
In view of \eqref{eq:cNorm} and \eqref{eq:moment_Cm}, we can compute the asymptotic behaviour of $Z_{n,\alpha}$ (and similar expressions) 
by extracting the coefficients of power series as in \eqref{eq:cNorm} and \eqref{eq:moment_Cm}.
One way to extract these coefficients is the saddle point method,  a standard tool in asymptotic analysis. 
The basic idea is to rewrite the 
expression \eqref{eq:relation to perms} as a complex contour integral and
choose the path of integration in a convenient way. 
The details of this procedure depend on the situation at hand 
and need to be done on a case by case basis. 
A general overview over the saddle-point method can be found 
in \cite[page~551]{FlSe09}. 
An important part of this computations is typically to find a solution of the so-called saddle-point equation.
%
%
%
%

We now treat the most general case of the 
saddle point method that is relevant for the present situation. 
Let $\bsq = (q_{j,n})_{1 \leq j \leq \alpha(n), n \in 
\N}$ be a triangular array. We assume that all 
$q_{j,n}$ 
are nonnegative, 
real numbers and  that for each $n\in\N$ there exists a $j$ such that $q_{j,n}>0$. 
We then
define $x_{n,\bsq}$ as the unique positive solution of
\begin{align}
n = \sum_{j=1}^{\alpha(n)} q_{j,n} x_{n,\bsq}^j.
\label{eq:GenSaddle}
\end{align} 
Let further
\begin{align}
\lambda_{p,n}
:=
\lambda_{p,n,\alpha,\boldsymbol{q}} :=\sum_{j=1}^{\alpha(n)} q_{j,n}j^{p-1}x_{n,\bsq}^j,
\label{eq:def_lambda_p}
\end{align} 
where $p\geq 1$ is a natural number. Due
to Equation \eqref{eq:GenSaddle},
\begin{equation}
\lambda_{p,n}\leq n\left(\alpha\left(n\right)\right)^{p-1}\label{eq:LambdaP}
\end{equation}
holds for all $p \geq 1$. 
%
%
We now define 
\begin{defn}
\label{def:admQ}A triangular array $\boldsymbol{q}$ is called admissible
if the following three conditions are satisfied:
\begin{enumerate}
    \item \label{enu:sadAppr}It satisfies
        \[
        \alpha\left(n\right)\log\left(x_{n,\boldsymbol{q}}\right)\approx\log\left(\frac{n}{\alpha\left(n\right)}\right).
        \]
    \item \label{enu:lam2Appr}We have
        \[
        \lambda_{2,n}\approx n\alpha\left(n\right).
        \]
    \item \label{enu:qNichtzuklein}
        There exist a non-negative sequence $\left(b_{n}\right)_{n\in\mathbb{N}}$
        and constants $\delta,c>0$ such that $b\left(n\right)/\alpha\left(n\right)<1-\delta$
        and $q_{j,n}\geq c>0$ for all $j\geq b\left(n\right)$ hold for $n$
        large enough.
\end{enumerate}
\end{defn}
Note that condition~\eqref{enu:sadAppr} implies in particular that 
$x_{n,\boldsymbol{q}}>1$  and that $ x_{n,\bsq} \to 1$ as $n\to\infty$. 
%
%
%
Let $B_r(0)$ denote the ball with center $0$ and 
radius $r$ in the complex plane. 
\begin{defn}
\label{def:admF}
Let $\boldsymbol{q}$ be an admissible triangular array. 
Then a sequence $\left(f_{n}\right)_{n\in\mathbb{N}}$ of functions is
called admissible (w.r.t. $\boldsymbol{q}$) if it satisfies the following
three conditions:
\begin{enumerate}
\item \label{enu:hol}
    There is $\delta>0$ such that $f_{n}$ is holomorphic
    on the disc $B_{x_{n,\boldsymbol{q}}+\delta}\left(0\right)$ if $n\in\mathbb{N}$
    is large enough.
\item \label{enu:FdSad}
    There exist constants $K,N>0$ such that
    \begin{equation}
    \sup_{z\in\partial B_{x_{n,\boldsymbol{q}}}\left(0\right)}\left|f_{n}\left(z\right)\right|\leq n^{K}\left|f_{n}\left(x_{n,\boldsymbol{q}}\right)\right|\label{eq:adm2}
    \end{equation}
    for all $n\geq N$.
\item \label{enu:Flieb}
    With the definition
    \begin{equation}
    |\!|\!|f_{n}|\!|\!|_{n}:=n^{-\frac{5}{12}}\left(\alpha\left(n\right)\right)^{-\frac{7}{12}}\sup_{\left|\varphi\right|\leq n^{-\frac{5}{12}}\left(\alpha\left(n\right)\right)^{-\frac{7}{12}}}\frac{\left|f_{n}^{\prime}\left(x_{n,\boldsymbol{q}}\mathrm{e}^{\mathrm{i}\varphi}\right)\right|}{\left|f_{n}\left(x_{n,\boldsymbol{q}}\right)\right|},
    \label{eq:adm3}
    \end{equation}
    we have $|\!|\!|f_{n}|\!|\!|_{n}\to0$ as $n\to\infty$.
\end{enumerate}
\end{defn}
%
%
%
%
We are now in the position to formulate our general
saddle point result.
\begin{prop}[{\cite[Proposition 3.2]{BeScZe17}}]
\label{prop:SPmethod} 
Let $\boldsymbol{q}$ be an admissible
triangular array and $\left(f_{n}\right)_{n\in\mathbb{N}}$ an admissible
sequence of functions. Then we have as $n\to\infty$
\begin{align}
\left[z^{n}\right]f_{n}\left(z\right)\exp\left(\sum_{j=1}^{\alpha\left(n\right)}\frac{q_{j,n}}{j}z^{j}\right)
=
\frac{f_{n}\left(x_{n,\boldsymbol{q}}\right)e^{\lambda_{0,n}}}{x_{n,\boldsymbol{q}}^{n}\sqrt{2\pi\lambda_{2,n}}}\left(1+\mathcal{O}\left(\frac{\alpha\left(n\right)}{n}+|\!|\!|f_{n}|\!|\!|_{n}\right)\right).
\label{eq:prop:GenSaddle}
\end{align}
\end{prop}
Note that the implicit constants in the $\caO(.)$  terms in \eqref{eq:prop:GenSaddle} can depend on $K$, $N$ and $\delta$ from the above definition of admissibility.
However, we require for our computations only the leading term in \eqref{eq:prop:GenSaddle}. Also we will not vary the values of $K$, $N$ and $\delta$.
Thus we need only the existence of $K$, $N$ and $\delta$, but not their values. We therefore can safely omit the dependence on $K$, $N$ and $\delta$. 

In view of Proposition~\ref{prop:SPmethod}, 
we see it is important to understand the asymptotic behaviour of $x_{n,\boldsymbol{q}}$ and $\lambda_{j,n}$ as $n\to\infty$.
Lemma \ref{lem:saddle_point_with_c} will be very useful for this purpose. 

\subsection{Proof or Proposition \ref{prop:LongestDiv}}
\label{sect:proof_longest_1}

We have by assumption $\mu_{\alpha(n)}\to \infty$. Thus we can apply Theorem~\ref{thm:Haupt}. 
We conclude that
\begin{align}
 \frac{C_{\alpha(n)}-\mu_{\alpha(n)}}{\sqrt{\mu_{\alpha\left(n\right)}\left(n\right)}}
 \stackrel{d}{\longrightarrow} N,
 \label{eq:CLT_for_long_cycles1}
\end{align}
where $N$ is a standard normal distributed random variable.
Since $\alpha(n)$ is the maximal cycle length, we have 
\begin{align*}
 \PTa{\left(\ell_{1},\ell_{2},\dots,\ell_{K}\right)  \neq  \big(\alpha\left(n\right),\alpha\left(n\right),\dots,\alpha\left(n\right)\big)}
 =
 \PTa{C_{\alpha\left(n\right)}<K}.
\end{align*}
Using \eqref{eq:CLT_for_long_cycles1}, we get
\begin{align*}
 \PTa{C_{\alpha\left(n\right)}<K}
 =
 \PTa{\frac{C_{\alpha(n)}-\mu_{\alpha(n)}}{\sqrt{\mu_{\alpha(n)}}}<\frac{K-\mu_{\alpha\left(n\right)}\left(n\right)}{\sqrt{\mu_{\alpha\left(n\right)}\left(n\right)}}}\xrightarrow{n\to\infty}0,
\end{align*}
and the claim follows.

\subsection{Proof of Proposition \ref{prop:LongestConv}}
\label{sect:proof_longest_2}

As a first step, we state
\begin{prop}
\label{prop:CCconWeak}
 Let $(m_k(n))_{n\in\N}$, $k=1,\ldots,K$,  be integer sequences satisfying $1 \leq m_k(n)\leq \alpha(n)$ and 
 $m_{k}(n) \neq m_l(n) $ for $k\neq l$. Suppose that
 \begin{align*}
  \mu_{m_k(n)} \to \mu_k \in[0,\infty[
 \end{align*}
 for all $k$. Then
 \begin{align}
  \left(C_{m_1(n)},\, \ldots,\, C_{m_K(n)}  \right)
  \stackrel{d}{\longrightarrow}
  \left(Y_1,\, \ldots,\, Y_K  \right)
 \end{align}
 where $(Y_k)_{k=1}^K$ a sequence of independent Poisson distributed random variables
 with parameters $\E{Y_k} = \mu_k$ for all $k=1,\ldots,K$.
\end{prop}
This proposition was proven in \cite{BeScZe17}, but we give the proof of this proposition for the case $K =1$ for the convenience of the reader.
\begin{proof}
Let $K=1$. We argue here with the moment generating function.
We saw in \eqref{eq:moment_Cm} that we have for $s\geq 0$
 \begin{align}
\ETa{e^{s C_{m_1(n)}}}
= 
\frac{1}{Z_{n,\alpha} }
\left[z^n\right]\exp\left(\vartheta (e^s-1)\frac{z^{m_1(n)}}{m_1(n)}\right) \exp\left(\vartheta\sum_{j=1}^{\alpha(n)}\frac{z^j}{j}\right).
\end{align}
We now apply Proposition~\ref{prop:SPmethod} to compute the asymptotic behaviour of this expression in the case $s\geq0$.
According to \cite{Ya11}, this is sufficient to prove the proposition.
We use $\bsq =(q_{j,n})$ with $q_j =\vartheta\,\mathbbm{1}_{\left\{j\leq \alpha(n)\right\}}$
and 
$f_n(z) =\exp\left(\vartheta (e^s-1)\frac{z^{m_1(n)}}{m_1(n)}\right)$.
We thus have to show that $\bsq$ and the sequence $(f_n)_{n\in\N}$ are admissible, see Definitions~\ref{def:admQ} and~\ref{def:admF}.
Inserting the definition of $\bsq$, we immediately get that the corresponding saddle point equation is given by \eqref{eq:StaSad}, 
hence the solution is $x_{n,\alpha}$. The admissibility of $\bsq$ then follows immediately from Lemma~\ref{lem:saddle_point_with_c}.
It remains to show that $(f_n)_{n\in\N}$ is admissible.
All $f_n$ are entire functions and hence we can choose any $\delta>0$. 
Since $s\geq 0$, we have for all $r>0$ and $\varphi\in[-\pi,\pi]$
\begin{align*}
 \left|f_n(r\mathrm{e}^{\mathrm{i}\varphi}) \right| \leq |f_n(r)|.
\end{align*}
Thus the second condition is fulfilled with $K=0$. 
For the third condition, we use
\begin{align*}
 f'_n(z) = \vartheta (e^s-1) z^{m_1(n)-1} f_n(z)
\end{align*}
and that $\mu_{m_1(n)} = \vartheta \frac{x_{n,\alpha}^{m_1(n)}}{m_1(n)}$. 
Inserting this and that $\mu_{m_1(n)} \to \mu_1$ immediately shows that the third condition is fulfilled. 
So we can apply Proposition~\ref{prop:SPmethod}. 
Using that $\ETa{e^{s C_{m_1(n)}}} =1$ for $s=0$, we obtain
 \begin{align}
\ETa{e^{s C_{m_1(n)}}}
\longrightarrow
\exp\big(\vartheta (e^s-1)\mu_1\big).
\end{align}
This completes the proof.
\end{proof}
Now we turn to the proof of Proposition \ref{prop:LongestConv}.
In this proof, we write $\mu_{\alpha(n)}(n)$ instead of $\mu_{\alpha(n)}$.
Let $j\in\mathbb{N}_{0}$ be arbitrary.
Using the definition of $\mu_{m}(n)$ in \eqref{eq:def_mu_n} with $m=\alpha(n)-j$, we get
\[
\frac{\mu_{\alpha\left(n\right)}\left(n\right)}{\mu_{\alpha\left(n\right)-j}\left(n\right)}
=
\frac{\alpha\left(n\right)-j}{\alpha\left(n\right)}\,x_{n,\vartheta}^{j}\xrightarrow{n\to\infty}1.
\]
Since $\mu_{\alpha(n)}(n)\to \mu$ by assumption, we get that
\[
\mu_{\alpha\left(n\right)-j}\left(n\right)\xrightarrow{n\to\infty}\mu
\]
for all $j\in\mathbb{N}_{0}$.
Proposition~\ref{prop:CCconWeak} therefore implies that the cycle counts $\left(C_{\alpha\left(n\right)-j}\right)_{0\leq j\leq d}$ converge in
distribution to a sequence $\left(Z_{j}\right)_{j=0}^{d}$, where $(Z_{j})_{j=0}^d$ is i.i.d. Poisson distributed with parameter $\mu$.
We now have, as $n\to\infty$,
\begin{align*}
\PTa{\ell_{k}\leq\alpha\left(n\right)-d}
=
\PTa{\sum_{i=0}^{d-1}C_{\alpha\left(n\right)-i}\leq k-1}
\to 
 \mathbb{P}\left[\sum_{i=0}^{d-1}Z_{i}\leq k-1\right].
\end{align*}
By the independence of $\left(Z_{j}\right)_{0\leq i\leq d}$,
the random variable $\sum_{i=0}^{d-1}Z_{i}$ is Poisson-distributed
with parameter $d\mu$. Thus,
\[
\Pb{\sum_{j=0}^{d-1}Z_{j}\leq k-1}
=
\sum_{j=0}^{k-1}\mathrm{e}^{-d\mu}\frac{\left(d\mu\right)^{j}}{j!}
=
\frac{1}{\Gamma\left(k\right)}\int_{d\mu}^{\infty}v^{k-1}\mathrm{e}^{-v}\mathrm{d}v,
\]
where $\Gamma(s)$ denotes the gamma function. 
The last equality follows by partial integration and induction. 
We now have
\[
\PTa{\ell_{k}=\alpha\left(n\right)-d}
=
\PTa{\ell_{k}\leq \alpha\left(n\right)-d} - \PTa{\ell_{k}\leq \alpha\left(n\right)-(d+1)}.
\]
This implies
\[
\PTa{\ell_{k}=\alpha\left(n\right)-d}
\xrightarrow{n\to\infty}
\frac{1}{\Gamma\left(k\right)}\int_{d\mu}^{\left(d+1\right)\mu}v^{k-1}\mathrm{e}^{-v}\mathrm{d}v.
\]
The claim is proved.
\begin{rem}
The proof of Proposition \ref{prop:LongestConv} can also be used to compute the limit of
\[
\PTa{\left(\ell_{k}\right)_{k=1}^{K}=\left(\alpha\left(n\right)-d_{k}\right)_{k=1}^{K}}
\]
 as $n$ tends to infinity since the event in question only depends
on a finite number of cycle counts $C_{\alpha\left(n\right)-j}$.
It is, however, cumbersome to provide a
closed form for such probabilities: The reason for this is that the
stochastic process $\left(\ell_{k}\right)_{k=1}^K$ is not Markovian, i.e. the
distribution of $\ell_{K}$ depends non-trivially on the distribution
of the random vector $\left(\ell_{k}\right)_{k=1}^{K-1}$. This is why
we only provide the readily interpretable results for one individual
$\ell_{k}$ at a time in the proposition.
\end{rem}

\subsection{Proof of Theorem \ref{thm:Longest0Poissonprocess}}
\label{sect:proof_longest_3}

%
We will first prove certain auxiliary results, assuming that $\mu_{\alpha(n)}\to 0$. 
Inserting the definition of $\mu_{\alpha(n)}$, see \eqref{eq:def_mu_n}, we get
%
%
%
\[
\mu_{d_{t}\left(n\right)}\left(n\right)
=
\vartheta\frac{x_{n,\alpha}^{d_{t}\left(n\right)}}{d_{t}\left(n\right)}
=
\vartheta\frac{(x_{n,\alpha})^{\alpha\left(n\right)-\left\lfloor t/\mu_{\alpha(n)}\right\rfloor }}{\alpha\left(n\right)-\left\lfloor t/\mu_{\alpha(n)}\right\rfloor }
=
\mu_{\alpha(n)}\,\frac{\alpha\left(n\right)}{\alpha\left(n\right)-\left\lfloor t/\mu_{\alpha(n)}\right\rfloor }x_{n,\alpha}^{-\left\lfloor t/\mu_{\alpha(n)}\right\rfloor }.
\]
We now have 
\[
\frac{\alpha\left(n\right)}{\alpha\left(n\right)-\left\lfloor t/\mu_{\alpha(n)}\right\rfloor }\xrightarrow{n\to\infty}1,
\]
locally uniformly in $t$ since $1/\mu_{\alpha(n)}=o\left(\alpha\left(n\right)\right)$
by Equation \eqref{eq:LongestConv0MuN}. 
By Lemma \ref{lem:saddle_point_with_c} and Equation \eqref{eq:LongestConv0MuN}, we have as $n\to\infty$
\begin{align}
x_{n,\alpha}^{-\left\lfloor t/\mu_{\alpha(n)}\right\rfloor }
=& 
\exp\left(-\left\lfloor \frac{t}{\mu_{\alpha(n)}}\right\rfloor \frac{1}{\alpha(n)}\log\left(\frac{n}{\vartheta\alpha\left(n\right)}\log\left(\frac{n}{\vartheta\alpha(n)}\right)\right) \big(1+o(1)\big)\right)\nonumber\\
= & 
\exp\left(\mathcal{O}\left(t\frac{\alpha\left(n\right)}{n}\right)\right)
\longrightarrow 1,
\label{eq:useful_for_thight}
\end{align}
locally uniformly in $t\geq0$. Altogether, we have locally uniformly in $t$ that
\[
\mu_{d_{t}\left(n\right)}\left(n\right)
\sim
\mu_{\alpha(n)}.
\]
Furthermore, the function $m\to \mu_{m}\left(n\right)$ is increasing for 
$m\geq \frac{\alpha(n)}{\log n}$.  
This follows by computing the derivative with respect to $m$ of $\mu_{m}\left(n\right)$ in \eqref{eq:def_mu_n} and using  Lemma~\ref{lem:saddle_point_with_c}.
%
%
%
We thus have locally uniformly in $t$
\begin{equation}
\sum_{m=d_t(n)+1}^{\alpha\left(n\right)}\mu_{m}\left(n\right)
\xrightarrow{n\to\infty}
\frac{t}{\mu_{\alpha(n)}}\mu_{\alpha(n)}
=
t.
\label{eq:LongestConv0zwischen}
\end{equation}

In order to establish convergence as a stochastic process, we begin
by proving convergence of the finite-dimensional distributions. 
More precisely, for $0=t_{0}\leq t_{1}<...<t_{K}$ and $K\in\mathbb{N}$,
consider the increments $\left(P_{t_{k}}-P_{t_{k-1}}\right)_{k=1}^{K}$.
%
We now have 
\begin{align}
 P_{t_{k}}-P_{t_{k-1}} = \sum_{j=d_{t_k}\left(n\right)+1}^{d_{t_{k-1}}\left(n\right)}C_{j}.
\end{align}
We begin by determining the moment generating function. 
We have 
\begin{align}
  &\ETa{\prod_{k=1}^{K}\exp\Big(s_{k}\left(P_{t_{k}}-P_{t_{k-1}}\right)\Big)}
 \nonumber \\
=& 
\frac{1}{Z_{n,\alpha,\vartheta}}
\left[z^{n}\right]
\exp\left(\sum_{k=1}^{K}\left(\mathrm{e}^{s_{k}}-1\right)\sum_{j=d_{t_k} +1}^{d_{t_{k-1}}}
\frac{\vartheta}{j}z^{j}\right)\exp\left(\sum_{j=1}^{\alpha\left(n\right)}\frac{\vartheta}{j}z^{j}\right),
\label{eq:LongestConv0MGF}
\end{align}
where $s_{k}\geq 0$  for all $1\leq k\leq K$. 
Equation~\eqref{eq:LongestConv0MGF} follows immediately with Lemma~\ref{lem:polya} and a small computation.
%
%
We will apply Proposition~\ref{prop:SPmethod} with $\bsq =(q_{j,n})$ with $q_j =\vartheta\,\mathbbm{1}_{\left\{j\leq \alpha(n)\right\}}$
and the perturbations
\[
f_{n}\left(z\right)
=
\exp\left(\sum_{k=1}^{K}\left(\mathrm{e}^{s_{k}}-1\right)\sum_{j=d_{t_k}+1}^{d_{t_{k-1}} }\frac{\vartheta}{j}z^{j}\right).
\]
To do this, we have to check that the array $\bsq$ and the sequence $(f_n)_{n\in\N}$ are admissible,
see Definitions~\ref{def:admQ} and~\ref{def:admF}.
The array  $\bsq$ is admissible by Lemma~\ref{lem:saddle_point_with_c}.
Let us now look at $(f_n)_{n\in\N}$.
The functions $f_{n}$ are entire. 
Thus we can use any $\delta>0$. 
Further, all coefficients of the Taylor expansion of $f_{n}(z)$ at $z=0$ are non-negative since all $s_{k}\geq0$.
This implies
\begin{align*}
 \left|f_{n}\left(z\right)\right|\leq f_{n}\left(x_{n,\vartheta}\right)
 \ \text{ for all }z\in\C \text{ with }\left|z\right|=x_{n,\vartheta}.
\end{align*}
It remains to check condition \eqref{eq:adm3}.
We have
\[
f_{n}^{\prime}\left(z\right)=\sum_{k=1}^{K}\left(\mathrm{e}^{s_{k}}-1\right)\sum_{j=d_{t_{k}} +1}^{d_{t_{k-1}} }\vartheta z^{j-1}f_{n}\left(z\right).
\]
We thus have for all $z\in\C$ with $\left|z\right|=x_{n,\vartheta}$
that
\begin{align}
 \left|\frac{f_{n}^{\prime}\left(z\right)}{f_{n}\left(x_{n,\vartheta}\right)}\right|
&\leq
\sum_{k=1}^{K}\left(\mathrm{e}^{s_{k}}-1\right)\sum_{j=d_{t_{k}} +1}^{d_{t_{k-1}} }\vartheta x_{n,\vartheta}^{j-1}\nonumber\\
&\leq 
\vartheta x_{n,\vartheta}^{\alpha(n)} \sum_{k=1}^{K}\left(\mathrm{e}^{s_{k}}-1\right)\sum_{j=\alpha\left(n\right)-\left\lfloor t_{k}/\mu_{\alpha(n)}\right\rfloor +1}^{\alpha\left(n\right)-\left\lfloor t_{k-1}/\mu_{\alpha(n)}\right\rfloor } 1.
\label{eq:upper_bounds_find_good_name}
\end{align}
Using the definition of $\mu_{\alpha(n)}$ in \eqref{eq:def_mu_n}, we see that we have locally uniformly in $s_k$ 
\begin{align*}
 \left|\frac{f_{n}^{\prime}\left(z\right)}{f_{n}\left(x_{n,\vartheta}\right)}\right|
=
\mathcal{O}\left( \vartheta x_{n,\vartheta}^{\alpha(n)}   \frac{t_{K}}{\mu_{\alpha(n)}} \right)
=
\mathcal{O}\left(\alpha(n) \right).
\end{align*}
Inserting this into \eqref{eq:adm3}, we obtain
\begin{align*}
    |\!|\!|f_{n}|\!|\!|_{n}
    &=
    n^{-\frac{5}{12}}\left(\alpha\left(n\right)\right)^{-\frac{7}{12}}
    \sup_{\left|\varphi\right|\leq n^{-\frac{5}{12}}\left(\alpha\left(n\right)\right)^{-\frac{7}{12}}}\frac{\left|f_{n}^{\prime}\left(x_{n,\boldsymbol{q}}\mathrm{e}^{\mathrm{i}\varphi}\right)\right|}{\left|f_{n}\left(x_{n,\boldsymbol{q}}\right)\right|}\\
    &\leq
    n^{-\frac{5}{12}}\left(\alpha\left(n\right)\right)^{-\frac{7}{12}} \mathcal{O}\left(\alpha(n) \right)
    =
    \mathcal{O}\left(\left(\frac{\alpha(n)}{n}\right)^{5/12} \right) \to 0.
\end{align*}
This implies that the sequence $(f_n)_{n\in\N}$ is admissible, so we
%
%
%
%
%
%
can apply Proposition~\ref{prop:SPmethod} to \eqref{eq:LongestConv0MGF}.
Observe that Equation~\eqref{eq:LongestConv0zwischen} entails
\[
\sum_{j=d_{t_{k}} +1}^{d_{t_{k-1}} }\mu_{j}(n)
=
\sum_{j=\alpha\left(n\right)-\left\lfloor t_{k}/\mu_{n}\right\rfloor +1}^{\alpha\left(n\right)-\left\lfloor t_{k-1}/\mu_{n}\right\rfloor }\mu_{j}\left(n\right)\xrightarrow{n\to\infty}
t_{k}-t_{k-1}
\]
for all $k$. 
Since we use for all $s_k$ the same array $\bsq$, including the case $s_1=\ldots=s_K=0$,
we get with Proposition~\ref{prop:SPmethod} that 
\begin{align*}
 \ETa{\prod_{k=1}^{K}\exp\Big(s_{k}\left(P_{t_{k}}-P_{t_{k-1}}\right)\Big)}
 &\sim
 f_{n}\left(x_{n,\vartheta}\right)
 =
\exp\left(\sum_{k=1}^{K}\left(\mathrm{e}^{s_{k}}-1\right)\sum_{j=d_{t_k}+1}^{d_{t_{k-1}} }\mu_{j}(n)\right)\\
&\longrightarrow
\sum_{k=1}^{K}\exp\left[\left(\mathrm{e}^{s_{k}}-1\right)\left(t_{k}-t_{k-1}\right)\right].
\end{align*}
%
%
%
%
This implies that the increments $ (P_{t_{k}}-P_{t_{k-1}})_{k=1}^K$
converge in distribution to independent random variables $\left(Z_{1},Z_{2},\dots,Z_{K}\right)$,
where $Z_{k}$ is Poisson-distributed with parameter $t_{k}-t_{k-1}$.
Thus the finite-dimensional distributions of $P_{t}$
converge weakly to the finite-dimensional distributions of the Poisson
process with parameter $1$.

To prove that the process $\{P_{t}, t\geq 0\}$ converges to the Poisson process with parameter $1$,
it remains to establish the tightness of the process $\{P_{t}, t\geq 0\}$.
By \cite[Theorem~13.5 and~(13.14)]{Bi99}, it is sufficient to show for each $T>0$ that 
\begin{align}
 \mathbb{E}_{n,\alpha}\left[\left(P_{t}-P_{t_{1}}\right)^{2}\left(P_{t_{2}}-P_{t}\right)^{2}\right]
=
\mathcal{O}\left(\left(t_{2}-t_{1}\right)^{2}\right)
\label{eq:tightness_criterium}
\end{align}

uniformly in $t,t_1,t_2$ with  $0\leq t_{1}\leq t\leq t_{2}\leq T$.
Note that we can assume that $\frac{t_{2}}{\mu_{\alpha(n)}}-\frac{t_{1}}{\mu_{\alpha(n)}} \geq 1$.
Otherwise $\left(P_{t}-P_{t_{1}}\right)^{2}\left(P_{t_{2}}-P_{t}\right)^{2} =0$
and the above equation is trivially fulfilled. 
%
%
%
Let $n$ be large enough such that $d_{T}\left(n\right)>0$.
By Equation \eqref{eq:LongestConv0MGF}, we have
\begin{align*}
 & \mathbb{E}_{n,\alpha}\left[\left(P_{t}-P_{t_{1}}\right)^{2}\left(P_{t_{2}}-P_{t}\right)^{2}\right]
=  
\left.\frac{\partial^{2}}{\partial s_{2}^{2}}\frac{\partial^{2}}{\partial s_{1}^{2}}
\mathbb{E}_{n,\alpha}\left[\mathrm{e}^{s_{1}\left(P_{t}-P_{t_{1}}\right)+s_{2}\left(P_{t_{2}}-P_{t}\right)}\right]\right|_{s_1=s_2=0}\\
= & \frac{1}{Z_{n,\alpha,\vartheta}}\left.\frac{\partial^{2}}{\partial s_{2}^{2}}\frac{\partial^{2}}{\partial s_{1}^{2}}\left[z^{n}\right]\exp\left(\left(\mathrm{e}^{s_{1}}-1\right)G_{n,t_{1},t}\left(z\right)+\left(\mathrm{e}^{s_{2}}-1\right)G_{n,t,t_{2}}\left(z\right)\right)\exp\left(\sum_{j=1}^{\alpha\left(n\right)}\frac{\vartheta}{j}z^{j}\right)\right|_{s_1=s_2=0}
\end{align*}
with $ G_{n,u,w}\left(z\right):=\sum_{j= d_w(n) +1}^{d_u(n) }\frac{\vartheta}{j}z^{j}$ for $0\leq u\leq w \leq T$.
Calculating the derivatives and entering $s_1=s_2=0$ gives
\begin{align*}
  \mathbb{E}_{n,\alpha}\left[\left(P_{t}-P_{t_{1}}\right)^{2}\left(P_{t_{2}}-P_{t}\right)^{2}\right]
  =
  \frac{1}{Z_{n,\alpha,\vartheta}}\left[z^{n}\right]g_{n}\left(z\right)\exp\left(\sum_{j=1}^{\alpha\left(n\right)}\frac{\vartheta}{j}z^{j}\right)
\end{align*}
with
\begin{align*}
g_{n}\left(z\right):=G_{n,t_{1},t}\left(z\right)\left(1+G_{n,t_{1},t}\left(z\right)\right)G_{n,t,t_{2}}\left(z\right)\left(1+G_{n,t,t_{2}}\left(z\right)\right).
\end{align*}
We now apply again Proposition~\ref{prop:SPmethod}. 
We use here the perturbations $(g_n)_{n\in\N}$ and as before  $\bsq =(q_{j,n})$ with $q_j =\vartheta\,\mathbbm{1}_{\left\{j\leq \alpha(n)\right\}}$.
Thus we only have to show that $(g_n)_{n\in\N}$ is admissible.
All $g_n$ are entire and we thus can use any $\delta>0$. 
Further the coefficients of the Taylor expansion of $g_n(z)$ at $z=0$ are all non-negative. 
Thus $\left|g_{n}\left(z\right)\right|\leq g_{n}\left(\left|z\right|\right)$
for all $z$.
It remains to check condition \eqref{eq:adm3}. 
We use here an estimate which is similar to the one in \eqref{eq:upper_bounds_find_good_name}.
We have for $z\in\C$ with $|z| = x_{n,\vartheta}$ that
\begin{align*}
  |G_{n,u,w}'\left(z\right)|
  &=
  \left|\sum_{j= d_w(n) +1}^{d_u(n) }\vartheta z^{j-1}\right|
  \leq 
  \vartheta \sum_{j= d_w(n) +1}^{d_u(n) }  x_{n,\vartheta}^{j-1}
  \leq 
   \vartheta  x_{n,\vartheta}^{\alpha(n)}\sum_{j=\alpha\left(n\right)-\left\lfloor w/\mu_{\alpha(n)}\right\rfloor +1}^{\alpha\left(n\right)-\left\lfloor u/\mu_{\alpha(n)}\right\rfloor } 1\\
  &=
     \vartheta  x_{n,\vartheta}^{\alpha(n)}  \left( \left\lfloor w/\mu_{\alpha(n)}\right\rfloor  -\left\lfloor u/\mu_{\alpha(n)}\right\rfloor \right).
\end{align*}
%
Similarly, we have
\begin{align}
  |G_{n,u,w}(x_{n,\vartheta})|
  &=
  \vartheta\sum_{j= d_w(n) +1}^{d_u(n) }  \frac{x_{n,\vartheta}^j}{j}
  \geq
  \frac{\vartheta}{d_u(n)}  x_{n,\vartheta}^{d_w(n)+1} \sum_{j= d_w(n) +1}^{d_u(n) } 1\nonumber\\
  &\geq
    \frac{\vartheta}{d_u(n)}  x_{n,\vartheta}^{d_w(n)+1}
    \left( \left\lfloor w/\mu_{\alpha(n)}\right\rfloor  -\left\lfloor u/\mu_{\alpha(n)}\right\rfloor       \right).
    \label{eq:useful_for_thight2}
\end{align}
Using \eqref{eq:useful_for_thight} and the definition of $d_w(n)$ in \eqref{eq:def_dt}, we get
\begin{align*}
 \left| \frac{G_{n,u,w}'\left(z\right)}{G_{n,u,w}(x_{n,\vartheta})} \right|
 \leq 
 d_u(n) x_{n,\vartheta}^{\left\lfloor w/\mu_{\alpha(n)}\right\rfloor +1}
 \leq \alpha(n) \exp\left(\mathcal{O}\left(T\frac{\alpha\left(n\right)}{n}\right)\right)
 = 
 \mathcal{O}\left( \alpha(n)\right).
\end{align*}
This estimate is uniform in $u,w$ with $0\leq u\leq w\leq T$.
Inserting this inequality into \eqref{eq:adm3} then gives
\begin{align*}
    |\!|\!|G_{n,u,w}|\!|\!|_{n}
    &\leq
    n^{-\frac{5}{12}}\left(\alpha\left(n\right)\right)^{-\frac{7}{12}} O\left(\alpha(n) \right)
    =
    O\left(\left(\frac{\alpha(n)}{n}\right)^{5/12} \right) \to 0.
\end{align*}
We thus have 
\begin{align}
 |\!|\!|g_{n}|\!|\!|_{n} \leq  2 |\!|\!|G_{n,t_1,t}|\!|\!|_{n} + 2 |\!|\!|G_{n,t,t_2}|\!|\!|_{n} 
 =  
 O\left(\left(\frac{\alpha(n)}{n}\right)^{5/12} \right).
 \label{eq_gn_triple_norm}
\end{align}
This estimate is uniform in $t,t_1,t_2$ with $0\leq t_1\leq t \leq t_2\leq T$.
This implies that the sequence $(g_n)_{n\in\N}$ is admissible.
Proposition~\ref{prop:SPmethod} then implies that
\begin{align*}
 \mathbb{E}_{n,\alpha}\left[\left(P_{t}-P_{t_{1}}\right)^{2}\left(P_{t_{2}}-P_{t}\right)^{2}\right]
= 
g_{n}\left(x_{n,\vartheta}\right) \left(1+\mathcal{O}\left(\frac{\alpha\left(n\right)}{n}+|\!|\!|g_{n}|\!|\!|_{n}\right)\right)
\leq 
2 g_{n}\left(x_{n,\vartheta}\right).
\end{align*}
Using the definition of $g_n$ and an estimate similar to \eqref{eq:useful_for_thight2}, we get
\begin{align*}
g_{n}\left(x_{n,\vartheta}\right)
\leq & 
\left(\sum_{j=d_{t_{2}\left(n\right)}+1}^{d_{t_{1}}\left(n\right)}\frac{\vartheta}{j}x_{n,\vartheta}^{j}\right)^{2}\left(1+\sum_{j=d_{t_{2}\left(n\right)}+1}^{d_{t_{1}}\left(n\right)}\frac{\vartheta}{j}x_{n,\vartheta}^{j}\right)^{2}\\
\leq & 
2\left(d_{t_{1}}\left(n\right)-d_{t_{2}}\left(n\right)\right)^{2}\mu_{\alpha(n)}^{2}\left(1+2\left(d_{t_{1}}\left(n\right)-d_{t_{2}}\left(n\right)\right)\mu_{\alpha(n)}\right)^{2}.
\end{align*}
Using the definition of $d_t(n)$ in \eqref{eq:def_dt} and that $0\leq t_1\leq t_2\leq T$, we obtain
\begin{alignat*}{1}
g_{n}\left(x_{n,\vartheta}\right)
\leq\, & 
2(1+2T)^2\big(d_{t_{1}}\left(n\right)-d_{t_{2}}\left(n\right)\big)^{2}\mu_{\alpha(n)}^{2}\\
=\,& 
2(1+2T)^2\left(\left\lfloor \frac{t_{2}}{\mu_{\alpha(n)}}\right\rfloor -\left\lfloor \frac{t_{1}}{\mu_{\alpha(n)}}\right\rfloor \right)^{2}\mu_{\alpha(n)}^{2}\\
\leq\, & 
2(1+2T)^2\left(\frac{t_{2}}{\mu_{\alpha(n)}}-\frac{t_{1}}{\mu_{\alpha(n)}}+1\right)^{2}\mu_{\alpha(n)}^{2}\\
\leq\, & 
8(1+2T)^2\left(t_{2}-t_{1}\right)^{2}.
\end{alignat*}
Note that we used for the last equation the assumption $\frac{t_{2}}{\mu_{\alpha(n)}}-\frac{t_{1}}{\mu_{\alpha(n)}}\geq 1$.
This shows that \eqref{eq:tightness_criterium} holds.
This completes the proof.
%
%

%
%

\subsection{Proof of Theorem \ref{thm:main_dtv}}
\label{sect:proof_dtv}
The proof follows mainly the ideas in \cite{ArTa92c}, where the case of uniform permutations is treated, and is also similar to the proof of Theorem~\ref{thm:main_thm2_old} in \cite{BeScZe17}. 

In order to establish Theorem~\ref{thm:main_dtv}, we have to introduce some notation.
We set
\begin{align}
 d_{b\left(n\right)}:=
 \|\bbP_{n,\vartheta, b(n),\alpha} - \widehat \bbP_{b(n)} \|_{\rm TV}.
\end{align}
Let $\left(Y_{j}\right)$ be as in Theorem~\ref{thm:main_dtv} and  set for $b_1$, $b_2\in\N$
\begin{equation}
T_{b_{1}b_{2}}^{(n)}
:=
\sum_{j=b_{1}+1}^{b_{2}} j Y_{j}.
\label{eq:Tdef}
\end{equation}
Further, let $\boldsymbol{C}_b=\left(C_1,C_2,\dots, C_{b(n)}\right)$ the vector of the cycle counts up to length $b(n)$, 
$\boldsymbol{Y}_b =\left(Y_1, Y_2, \dots, Y_{b(n)} \right)$, and $\boldsymbol{c}=\left(c_1,c_2,\dots,c_{b(n)}\right)\in\N^{b(n)}$ a vector. 
We then have for all $\boldsymbol{c}$
\begin{align}
 \PTa{\boldsymbol{C}_{b}=\boldsymbol{c}}
=
\Pb{\left.\boldsymbol{Y}_{b}=\boldsymbol{c}\right|T_{0\alpha\left(n\right)}^{\left(n\right)}=n}.
\label{eq:cond_relation_dtv}
\end{align}
The proof of this equality is the same as for the uniform measure on $S_n$ in \cite{ArTa92c} and we thus omit it.
As in \cite[Section~4.2]{BeScZe17}, one can use \eqref{eq:cond_relation_dtv} to show that
\begin{align}
 d_{b\left(n\right)}
=
\sum_{r=0}^{\infty}\Pb{T_{0b(n)}^{\left(n\right)}=r} \left(1-\frac{\mathbb{P}\left[T_{b(n)\alpha(n)}^{\left(n\right)}=n-r\right]}{\mathbb{P}\left[T_{0\alpha\left(n\right)}^{\left(n\right)}=n\right]}\right)_{+},
\label{eq:dtv_with_Tn}
\end{align}
where $(y)_+ = \max(y,0)$.
%
%
We will split this sum into pieces. 
%
%
%
%
%
%
We have 
\[
d_{b\left(n\right)}
\leq
\mathbb{P}\left[T_{0b}^{\left(n\right)}\geq \rho \E{T_{0b(n)}^{\left(n\right)}}\right]
+
\max_{1\leq r \leq \rho \E{T_{0b(n)}^{\left(n\right)}}}
\left(1-\frac{\mathbb{P}\left[T_{b\alpha(n)}^{\left(n\right)}=n-r\right]}{\mathbb{P}\left[T_{0\alpha\left(n\right)}^{\left(n\right)}=n\right]}\right)_{+},
\]
where $\rho= \rho(n) > 1$ is arbitrary. We now have 
\begin{lem}
\label{lem:rho}
Let $\rho>1$. Then,
\[
\Pb{T_{0b}^{\left(n\right)}\geq\rho \E{T_{0b(n)}^{\left(n\right)}}}
\leq
\exp\left(\E{T_{0b(n)}^{\left(n\right)}} \frac{\rho-\rho\log(\rho)}{b(n)}\right).
\]
\end{lem}
\begin{proof}
We set $m:=\E{T_{0b(n)}^{\left(n\right)}}$.
We then have for all $s\geq0$
\begin{align}
 \Pb{T_{0b(n)}^{\left(n\right)}\geq\rho m }
 =
 \Pb{\mathrm{e}^{sT_{0b(n)}^{\left(n\right)}}\geq \mathrm{e}^{s\rho m}} 
 \leq 
 \frac{\E{ \mathrm{e}^{sT_{0b}^{\left(n\right)}}}}{\mathrm{e}^{s\rho m}}.
 \label{eq:markov_T0b}
\end{align}
The independence of the $Y_j$ and $m=\sum_{j=1}^{b(n)} j \mu_j(n) = \vartheta\sum_{j=1}^b x_{n,\vartheta}^j$ imply that 
\begin{align}
 \log\left( \E{ \mathrm{e}^{sT_{0b}^{\left(n\right)}}} \right)
 &=
 \sum_{j=1}^b{\mu_j(n)}(\mathrm{e}^{js}-1)
 =
 \vartheta\sum_{j=1}^b x_{n,\vartheta}^j\int_0^{s} e^{jx}dx
 \leq
 \vartheta\sum_{j=1}^b x_{n,\vartheta}^j\int_0^{s} e^{bx}dx\nonumber\\
 &\leq 
 m\int_0^{s} e^{bx}dx
 =
 m\frac{e^{bs}-1}{b}
 \leq
 \frac{me^{bs}}{b}.
 \label{eq:markov_T0b:in}
\end{align}
We thus have $\Pb{T_{0b}^{\left(n\right)}\geq\rho m}
\leq \exp\left( \frac{me^{bs}}{b} - s\rho m \right)$.
%
We now use $s=\frac{1}{b}\log\left(\rho\right)$, which is by assumption non-negative.
Inserting this into the above inequality completes the proof.
\end{proof}
In order to choose a suitable $\rho$, we have to determine the asymptotic behavior of $\E{T_{0b(n)}^{(n)}}$.
Using the definition of $\mu_j(n)$ in \eqref{eq:def_mu_n}, we get
\begin{align}
 \E{T_{0b(n)}^{(n)}}
 =
 \sum_{j=1}^{b(n)} j \mu_j(n)
 =
 \vartheta \sum_{j=1}^{b(n)} (x_{n,\alpha})^j
 =
 \vartheta x_{n,\alpha} \frac{(x_{n,\alpha})^{b(n)}-1 }{x_{n,\alpha} -1}.
\end{align}
We know from Lemma~\ref{lem:saddle_point_with_c} that $x_{n,\alpha} \to 1$ and
\begin{align}
 (x_{n,\alpha})^{b(n)} 
 \sim
 \left(\frac{n}{\vartheta \alpha(n)} \log\left(\frac{n}{\vartheta \alpha(n)}\right) \right)^{b(n)/\alpha(n)}.
\end{align}
If $b(n)= o\left(\alpha(n)/\log(n)\right)$ then $(x_{n,\alpha})^{b(n)} \to 1$ and thus $\E{T_{0b(n)}^{(n)}}\sim b(n)$.
However, we can also have $b(n)\geq  c\frac{\alpha(n)}{\log(n)}$ for some $c>0$.
Using that $x_{n,\alpha} -1 \sim \log(x_{n,\alpha})$, we immediately obtain
\begin{align}
 \E{T_{0b(n)}^{(n)}} 
 &\approx
 \frac{\alpha(n)}{\log(n)}  \left(\frac{n}{\vartheta \alpha(n)} \log\left(\frac{n}{\vartheta \alpha(n)}\right) \right)^{b(n)/\alpha(n)}.
 \label{eq:E(Tb)}
\end{align}
This implies that we have for $n$ large
\begin{align}
\frac{\alpha(n)}{\log(n)}
 \leq 
 \E{T_{0b(n)}^{(n)}} 
 \leq 
 \frac{\alpha(n)n^\epsilon}{\log(n)},
  \label{eq:E(Tb)01}
\end{align}
where $\epsilon>0$ can be chosen arbitrarily.
In view of \eqref{eq:E(Tb)01} and $b(n)=o(\alpha(n))$, we use $\rho =\log^2(n)$ in Lemma~\ref{lem:rho}. 
With this choice of $\rho$, we immediately get that $\Pb{T_{0b}^{\left(n\right)}\geq\rho \E{T_{0b(n)}^{\left(n\right)}}} = \caO(n^{-A})$ where $A>0$ is arbitrary.
%
%
Inserting this into \eqref{eq:dtv_with_Tn}, with $A=2$, we get 
\begin{align}
 d_{b\left(n\right)}
\leq
\max_{r\leq \rho \E{T_{0b(n)}^{\left(n\right)}}} \left(1-\frac{\mathbb{P}\left[T_{b(n)\alpha(n)}^{\left(n\right)}=n-r\right]}{\mathbb{P}\left[T_{0\alpha\left(n\right)}^{\left(n\right)}=n\right]}\right)_{+}
+
\caO(n^{-2}).
\label{eq:dtv_with_Tn2}
\end{align}
We look next at $T_{b(n)\alpha(n)}^{\left(n\right)}$.
Using that that all $Y_j$ are independent, we get that the probability generating function of $T_{b\alpha(n)}^{\left(n\right)}$ is
\begin{align}
 \E{z^{T_{b\alpha(n)}^{\left(n\right)}}}
 =
 \exp\left(\sum_{j=b+1}^{\alpha(n)} \mu_j(n) (z^j-1)\right).
\end{align}
Using that $\mu_j(n) = \vartheta \frac{\left(x_{n,\alpha}\right)^{j}}{j}$, we get
\begin{align*}
 \Pb{T_{b(n)\alpha(n)}^{\left(n\right)}=n-r}
 &=
  \exp\left(-\sum_{j=b+1}^{\alpha(n)}\mu_j(n) \right) x_{n,\alpha}^{n-r}
 \left[z^{n}\right]z^r\exp\left(\vartheta\sum_{j=b+1}^{\alpha(n)}\frac{1}{j} z^{j}\right).
\end{align*}
Similarly, we obtain
\begin{align*}
\Pb{T_{0\alpha\left(n\right)}^{\left(n\right)}=n}
= 
\exp\left(-\sum_{j=1}^{\alpha(n)}\mu_j(n) \right)
x_{n,\alpha}^{n}
\left[z^{n}\right]\exp\left(\vartheta\sum_{j=1}^{b}\frac{1}{j}z^{j}\right)\exp\left(\vartheta   \sum_{j=b+1}^{\alpha(n)}\frac{1}{j}z^{j}\right).
\end{align*}
Thus we have to determine for $r\leq \rho \E{T_{0b(n)}^{\left(n\right)}}$ the asymptotic behaviour of
\begin{align*}
  \left[z^{n}\right]z^r\exp\left(\vartheta\sum_{j=b+1}^{\alpha(n)}\frac{1}{j} z^{j}\right)
  \ \text{ and } \ 
  \left[z^{n}\right]\exp\left(\vartheta\sum_{j=1}^{b}\frac{1}{j}z^{j}\right)\exp\left(\vartheta\sum_{j=b+1}^{\alpha(n)}\frac{1}{j}z^{j}\right).
\end{align*}
We do this with Proposition~\ref{prop:SPmethod}. 
We use for both the triangular array 
\begin{align}
\bsq = (q_{j,n})_{1 \leq j \leq \alpha(n), n \in \N} \ \text{ with } \ 
 q_{j,n} =\vartheta\,\mathbbm{1}_{\left\{b(n)+1\leq j\leq \alpha(n)\right\}}.
 \label{eq:q_for_dtv}
\end{align}
Furthermore, we use the perturbations $f_{1,n}(z)=z^{r}$ for the first and $f_{2,n}(z)=\exp\left(\vartheta\sum_{j=1}^{b}\frac{1}{j}z^{j}\right)$ for the second expression.
We thus have to show that $\bsq$ and $f_{1,n}(z)$ and $f_{2,n}(z)$ are admissible, see Definitions~\ref{def:admQ} and~\ref{def:admF}.
We now have
\begin{lem}
\label{lem:NeuerSattel}
Let $b=o(\alpha(n))$ and define $x_n$ to be the solution of the equation
\begin{align}
 n=\vartheta\sum_{j=b+1}^{\alpha}x_{n}^{j}.
 \label{eq:def_xn}
\end{align}
We then have $x_{n,\alpha} \leq x_{n} \leq x_{n,\alpha-b}$ and $|x_{n}-x_{n,\alpha}|=\mathcal{O}\left(\frac{1}{\alpha(n)}\right)$. 
Furthermore the triangular array $\bsq$ in \eqref{eq:q_for_dtv} is admissible. 
\end{lem}

\begin{proof}
We have by definition that $x_{n,\alpha} \leq x_{n}$. 
Further, $x_{n,\alpha-b}$ is the solution of 
\[
n=\sum_{j=1}^{\alpha(n)-b} (x_{n,\alpha-b})^{j}.
\]
Since $\alpha(n)<n$, we have $x_{n}\geq 1$ and $x_{n,\alpha-b}\geq 1$.
This implies that $x_{n}\leq x_{n,\alpha-b}$.
Lemma~\ref{lem:saddle_point_with_c} now implies
\[
|x_{n}-x_{n,\alpha}| 
=
x_{n}-x_{n,\alpha}
\leq 
x_{n,\alpha-b}-x_{n,\alpha}
=
\mathcal{O}\left(\frac{1}{\alpha(n)}\right).
\]
Further, $x_{n,\alpha}$ and $x_{n,\alpha-b}$ are admissible by Lemma~\ref{lem:saddle_point_with_c}.
Thus $x_{n,\alpha} \leq x_{n} \leq x_{n,\alpha-b}$ together with Equation~\eqref{eq:StaSadAs} 
immediately shows that $ x_{n}$ fulfills Condition~(1) in Definition~\ref{def:admQ}.
Furthermore, we also get 
\begin{align}
 \log(x_n) \approx \frac{\log(n)}{\alpha(n)} 
 \ \text{ and } \ 
 x_n -1 \approx  \frac{\log(n)}{\alpha(n)}.
 \label{eq:approx_xn}
\end{align}
To see that $ x_{n}$ fulfills Condition~(2), one uses \eqref{eq:approx_xn} and the identity
\begin{align}
 \sum_{j=0}^d jq^j =  \frac{d q^{d+1}}{q-1} -\frac{q(q^d-1)}{(q-1)^2} 
 \text{ for all }d\in\N, q\neq 0.
\end{align}
Condition~(3) is obvious. 
Thus $\bsq$ is admissible.
\end{proof}
We now can show 
\begin{lem}
\label{lem:NeuerSattel2}
The sequences $(f_{1,n})_{n\in\N}$ with $f_{1,n}=z^r$ is admissible for all $ r=o\left(n^{\frac{5}{12}}\alpha^{\frac{7}{12}}\right)$.
Further, $(f_{2,n})_{n\in\N}$ with $f_{2,n}=\exp\left(\vartheta\sum_{j=1}^{b}\frac{1}{j}z^{j}\right)$ is admissible.
\end{lem}
\begin{proof}
We start with $(f_{1,n})_{n\in\N}$. 
Since all $f_{1,n} =z^r$, the first two conditions of Definition~\ref{def:admF} are fulfilled with $\delta=N=1$ and $K=0$ for all $r$.
 We now have
\begin{align*}
 |\!|\!|f_{1,n}|\!|\!|_{n}
 \leq 
n^{-\frac{5}{12}}\left(\alpha\left(n\right)\right)^{-\frac{7}{12}} rx_n^{-1}.
\end{align*}
Since $x_n\to 1$, we have $|\!|\!|f_{1,n}|\!|\!|_{n}\to 0$ if and only if $r= o(n^{\frac{5}{12}}\left(\alpha\left(n\right)\right)^{\frac{7}{12}})$.
This completes the proof of the first half of the statement.
For $(f_{2,n})_{n\in\N}$, we also have only to check the third condition.
Lemma~\ref{lem:saddle_point_with_c} implies that $x_n-1 \geq c\log(n)/\alpha(n)$ for some $c>0$. 
Since $x_{n,\alpha} \leq x_n\leq x_{n,\alpha-b}$ and  $b=o(\alpha(n))$, we get with Lemma~\ref{lem:saddle_point_with_c}
\begin{align*}
 \frac{|f_{2,n}^{\prime}(z)|}{|f_{2,n}(x_n)|}
 \leq
 \sum_{j=0}^{b-1} x_{n}^{j}
 =
 \frac{ x_{n,\alpha-b}^{b}-1}{x_{n,\alpha}-1}
 = 
 \caO(n^\epsilon \alpha(n))
 \ \text{ for all $z$ with }|z|=x_n,
\end{align*}
where $\epsilon>0$ can be chossen arbitarily small. 
We thus have $ |\!|\!|f_{2,n}|\!|\!|_{n} \leq n^{-\frac{5}{12}+\epsilon}(\alpha(n))^{\frac{5}{12}}$.
Since $\alpha(n)\leq n^{a_2}$ with $a_2<1$, we see that $|\!|\!|f_{2,n}|\!|\!|_{n}\to 0$ for $\epsilon>0$ small enough.
\end{proof}
We know from \eqref{eq:E(Tb)01} that $\E{T_{0b(n)}^{(n)}} \leq \frac{\alpha(n)n^\epsilon}{\log(n)}$ for each $\epsilon>0$ and $n$ large enough.
This shows that we can use Proposition~\ref{prop:SPmethod} to compute $\Pb{T_{b(n)\alpha(n)}^{\left(n\right)}=n-r}$
and $\Pb{T_{0\alpha(n)}^{\left(n\right)}=n}$ for $r\leq \rho \E{T_{0b(n)}^{\left(n\right)}}$.
We thus have 
\begin{align}
\frac{\Pb{T_{b(n)\alpha(n)}^{\left(n\right)}=n-r}}{\Pb{T_{0\alpha(n)}^{\left(n\right)}=n}}
=
x_{n,\alpha}^{-r}x_{n}^{r} \exp\left(-\vartheta\sum_{j=1}^{b}\frac{1}{j}\left(x_{n}^{j}-x_{n,\alpha}^{j}\right)\right)
\left(1 + R_n\right),
\label{eq:dtv_estimate_allmost_finished0}
\end{align}
where
\begin{align*}
	R_n 
	&= 
	 \caO\left(\frac{\alpha\left(n\right)}{n} +|\!|\!|f_{1,n}|\!|\!|_{n}+ |\!|\!|f_{2,n}|\!|\!|_{n} \right).
\end{align*}
Note  that the implicit constant in the error term in Proposition~\ref{prop:SPmethod} only depends on the used
$K$, $N$ and $\delta$. 
Since we use for each $r\leq \rho \E{T_{0b(n)}^{\left(n\right)}}$ the same $K$, $N$ and $\delta$, we get that $R_n$ is uniform in $r$.
We now have to distinguish the two cases  $b(n) =o(\alpha(n))$ and $b(n) =o\left(\frac{\alpha(n)}{\log(n)}\right)$
for the error terms in \eqref{eq:thm:main_dtv1} and \eqref{eq:thm:main_dtv2}. 
In the case $b(n) =o(\alpha(n))$, we get with \eqref{eq:E(Tb)01} and the proof of Lemma~\ref{lem:NeuerSattel2} that
$$
|\!|\!|f_{1,n}|\!|\!|_{n}
\leq 
\frac{r }{n^{\frac{5}{12}}(\alpha(n))^{\frac{7}{12}}} 
\leq 
\frac{\rho \E{T_{0b(n)}^{\left(n\right)}} }{n^{\frac{5}{12}}(\alpha(n))^{\frac{7}{12}}} 
=
\caO\left(n^{\epsilon} \left(\frac{\alpha(n)}{n}\right)^{\frac{5}{12}}\right)
$$
for each $ \epsilon>0$.  We thus have that $R_n$ is as in \eqref{eq:thm:main_dtv1}. 
In the case $b(n) =o\left(\frac{\alpha(n)}{\log(n)}\right)$, we have $\E{T_{0b(n)}^{\left(n\right)}} \sim b(n)$.
Using this, we immediately get that $R_n$ is as in \eqref{eq:thm:main_dtv2}. 

It thus remains to compute the asymptotic behaviour of the main term in \eqref{eq:dtv_estimate_allmost_finished0}.
We thus need an estimate for $x_n^b - x_{n,\alpha}^b$.
Unfortunately, the bounds obtained from the Lemmas~\ref{lem:NeuerSattel} and~\ref{lem:saddle_point_with_c} are not strong enough.
%
%
To overcome this issue, let us consider for $y\in\R$ the equation
\begin{align}
\vartheta e^{\alpha(n) y} =ny.
\label{eq:def_y_n_alpha}
\end{align}
It is straightforward to see that this equation has for $\frac{n}{\vartheta\alpha(n)}>e$ two solutions.
We denote these by $y_{n,\alpha,0}$ and $y_{n,\alpha}$ with $0<y_{n,\alpha,0}<y_{n,\alpha}$.
It is straightforward to see that $y_{n,\alpha,0}\sim\frac{\vartheta}{n}$ and $y_{n,\alpha}\sim\frac{\log (n/\alpha(n))}{\alpha(n)}$ as $n\to\infty$.
We have
\begin{lem}
\label{lem:NeuerSattel3}
We have 
\begin{align}
 \alpha(n)\,y_{n,\alpha} 
 =
 \log\left(\frac{n}{\vartheta\alpha(n)}\log\left(\frac{n}{\vartheta\alpha\left(n\right)}\right)\right) 
 + 
 O\left( \frac{\log\log(n)}{\log(n)}\right).
 \label{eq_y_n_asympt}
\end{align}
Furthermore, we have for $b=o(\alpha(n))$ that
\begin{align}
 \log(x_{n,\alpha}) = y_{n,\alpha} + \caO\left(\frac{1}{n\log(n)}\right)
 \ \text{ and } \
 \log(x_{n}) = y_{n,\alpha} + \caO\left(\frac{e^{by_{n,\alpha}}}{n\log(n)}\right).
 \label{eq:x_n_with_y_n_alpha}
\end{align}
\end{lem}
We first complete our computations of the main term in \eqref{eq:dtv_estimate_allmost_finished0} with Lemma~\ref{lem:NeuerSattel3} 
and then give the proof of Lemma~\ref{lem:NeuerSattel3}.
We have 
\begin{align}
 \vartheta\sum_{j=1}^{b}\frac{1}{j}\left(x_{n}^{j}-x_{n,\alpha}^{j}\right)
 &=
 \vartheta \sum_{j=0}^{b-1}\int_{x_{n,\alpha}}^{x_{n}}v^{j}\mathrm{d}v
 = 
 \vartheta \int_{x_{n,\alpha}}^{x_{n}}\frac{v^{b}-1}{v-1}\mathrm{d}v
 \leq 
 \frac{\vartheta}{x_{n,\alpha} -1} \int_{x_{n,\alpha}}^{x_{n}}v^{b}\mathrm{d}v\nonumber\\
 &=
  \frac{\vartheta}{x_{n,\alpha} -1} \left( \frac{(x_n)^{b+1}}{b+1} - \frac{(x_{n,\alpha})^{b+1}}{b+1}\right).
  \label{eq:dtv_estimate_allmost_finished}
\end{align}
We use \eqref{eq:x_n_with_y_n_alpha} and get for some $\epsilon>0$
\begin{align*}
(x_n)^{b+1} - (x_{n,\alpha})^{b+1}
&=
(x_{n,\alpha})^{b+1}
\left(\exp\Big( (b+1) (\log x_{n} - \log x_{n,\alpha}) \Big) - 1 \right)\\
&=
(x_{n,\alpha})^{b+1}
\left(\exp\left( (b+1) \caO\left(\frac{e^{by_{n,\alpha}}}{n\log(n)}\right) \right) - 1 \right)\\
& =
(x_{n,\alpha})^{b+1} (b+1) \caO\left(\frac{e^{by_{n,\alpha}}}{n\log n}\right).
\end{align*}
%
Equation~\ref{eq_y_n_asympt} and  Lemma~\ref{lem:saddle_point_with_c} imply that 
$e^{by_{n,\alpha}}=  \caO\left( n^{\epsilon} \right)$ and $(x_{n,\alpha})^{b+1} =  \caO\left( n^{\epsilon} \right)$,
%
where $\epsilon>0$ can be chosen arbitrarily small. 
Using this  and \eqref{eq:approx_xn}, we get
\begin{align*}
 \vartheta\sum_{j=1}^{b}\frac{1}{j}\left(x_{n}^{j}-x_{n,\alpha}^{j}\right)
 &=  
 \caO\left(\frac{(b+1)e^{by_{n,\alpha}} (x_{n,\alpha})^{b+1} }{n\log^2(n)}\right)
 =
 \caO\left(\frac{b+1}{n^{1-2\epsilon}\log^2(n)}\right).
\end{align*}
%
Inserting this into \eqref{eq:dtv_estimate_allmost_finished} gives
\begin{align*}
\frac{\Pb{T_{b(n)\alpha(n)}^{\left(n\right)}=n-r}}{\Pb{T_{0\alpha(n)}^{\left(n\right)}=n}}
&\geq 
\exp\left(  \caO\left(\frac{b+1}{n^{1-2\epsilon}\log^2(n)}\right)\right) \left(1+ R_n\right)\\
&=
1+\caO\left(n^{\epsilon} \left(\frac{\alpha(n)}{n}\right)^{\frac{5}{12}}\right).
\end{align*}
This equation together with \eqref{eq:dtv_with_Tn2} completes the proof of Theorem~\ref{thm:main_dtv}.

\begin{proof}[Proof of Lemma~\ref{lem:NeuerSattel3}]
We start with \eqref{eq_y_n_asympt}.
We insert the approach 
$$y = \frac{1}{\alpha(n)}\log\left(\frac{n}{\vartheta\alpha(n)}\log\left(\frac{n}{\vartheta\alpha\left(n\right)}\right)\right) +v$$
with $v\in\R$ into \eqref{eq:def_y_n_alpha}. 
This leads to the equation
\begin{align}
 \log\left(\frac{n}{\vartheta\alpha(n)}\right) e^{\alpha(n) v}
 =
 \log\left(\frac{n}{\vartheta\alpha(n)}\log\left(\frac{n}{\vartheta\alpha\left(n\right)}\right)\right) + \alpha(n)v.
  \label{eq:reform_saddle_equation0}
\end{align}
Note that we have 
\begin{align}
 \log(y) \leq \log(y\log(y)) \leq (1+\epsilon) \log(y) 
 \label{eq:bound_for_logs}
\end{align}
for all $\epsilon>0$ and $y$ large enough.
Using this, it is straightforward to see that equation \eqref{eq:reform_saddle_equation0} has exactly one solution in the region $v\geq 0$ and that this solution has to be $o\left(\frac{1}{\alpha(n)}\right)$ as $n\to\infty$. To obtain a lower bound for $v$, we use the inequality $e^{x}\leq 1+2x$ for $0\leq x\leq \log 2$. 
Thus $v$ is larger than the solution $v'$ of the equation 
\begin{align}
 \log\left(\frac{n}{\vartheta\alpha(n)}\right) (1+2 \alpha(n) v')
 =
 \log\left(\frac{n}{\vartheta\alpha(n)}\log\left(\frac{n}{\vartheta\alpha\left(n\right)}\right)\right) + \alpha(n)v'.
\end{align}
A simple computation gives 
\begin{align}
 v'= \frac{\log\log\left(\frac{n}{\vartheta\alpha(n)}\right) }{2\alpha(n)\log\left(\frac{n}{\vartheta\alpha(n)}\right) +\alpha(n)}
\end{align}
This establishes a lower bound for $v$. For an upper bound, we argue similarly with $1+x \leq e^{x}$ for $x\geq 0$. 
This completes the proof of \eqref{eq_y_n_asympt}.

We prove \eqref{eq:x_n_with_y_n_alpha} only for $x_n$. 
The asymptotics for $x_{n,\alpha}$ then follows immediately by inserting $b=0$ into the asymptotics for $x_n$.
The defining equation \eqref{eq:def_xn} of $x_n$ has exactly one solution can be rewritten as
\begin{align}
 \vartheta(x_n)^{\alpha(n)} -\vartheta(x_n)^{b}
 =
 n\left(1-(x_n)^{-1} \right).
\end{align}
We now insert $x_n = e^y$. This gives 
\begin{align}
 \vartheta e^{\alpha(n)y} -\vartheta e^{by}
 =
 n\left(1-e^{-y} \right).
 \label{eq:reform_saddle_equation}
\end{align}
The equation \eqref{eq:reform_saddle_equation} has exactly one solution in the region $y>0$.
Further, both sides of \eqref{eq:reform_saddle_equation} are monotone increasing functions of $y$.
Inserting $y=y_{n,\alpha} \pm \frac{c}{\alpha(n)}$ with $c>0$ into \eqref{eq:reform_saddle_equation} and using \eqref{eq_y_n_asympt} shows that
the RHS of \eqref{eq:reform_saddle_equation} behaves like
\begin{align}
 n\left(1-e^{-y_{n,\alpha} \pm \frac{c}{\alpha(n)}} \right) 
 \sim
 \frac{n}{\alpha(n)}\log\left(\frac{n}{\vartheta\alpha(n)}\log\left(\frac{n}{\vartheta\alpha\left(n\right)}\right)\right).
\end{align}
On the other hand, the LHS of \eqref{eq:reform_saddle_equation} behaves like
\begin{align}
 \vartheta e^{\alpha(n)\left(y_{n,\alpha} \pm \frac{c}{\alpha(n)}\right)} -\vartheta e^{b\left(y_{n,\alpha} \pm \frac{c}{\alpha(n)}\right)}
 \sim 
 e^{\pm c}\frac{n}{\alpha(n)}\log\left(\frac{n}{\vartheta\alpha(n)}\right).
\end{align}
Using \eqref{eq:bound_for_logs}, we immediately see that the solution of \eqref{eq:reform_saddle_equation}
has to be in the interval $[y_{n,\alpha} - \frac{c}{\alpha(n)}, y_{n,\alpha} + \frac{c}{\alpha(n)}]$.
We now use the approach $y = y_{n,\alpha} +v$.
Clearly, we must have $v=o\left(\frac{1}{\alpha(n)}\right)$. 
We now argue as for \eqref{eq_y_n_asympt}.
To get a lower bound for $v$, we use  $1+x\leq e^x$ and  $1-e^{-x}\leq x$.
This leads to the equation 
\begin{align*}
 \vartheta e^{\alpha(n) y_{n,\alpha}} (1 +\alpha(n) v') - \frac{3}{2}\vartheta e^{by_{n,\alpha}} 
 =
  n(y_{n,\alpha} +v').
\end{align*}
Using the definition of $y_{n,\alpha}$ in \eqref{eq:def_y_n_alpha}, we immediately get 
\begin{align*}
 v'
= 
\frac{3\vartheta e^{by_{n,\alpha}}}{2\vartheta e^{\alpha(n) y_{n,\alpha}}\alpha(n)- 2n}
=
\frac{3\vartheta e^{by_{n,\alpha}}}{2 n y_{n,\alpha} \alpha(n)- 2n}
\sim
\frac{3\vartheta e^{by_{n,\alpha}}}{2 n \log(n)}.
\end{align*}
The upper bound is obtained similarly. This completes the proof.

\end{proof}

\subsection{Proof of Theorem \ref{thm:Haupt}}
\label{sect:proof_clt}
%

We give here the proof for the case $K=1$ only. 
We thus write $m(n)$ and $\mu_{m(n)}$ instead of $m_1(n)$ and $\mu_{m_1(n)}(n)$. 
This mainly simplifies the notation, but does not change the argument used.
As in \cite{BeScZe17}, the proof will be based upon point-wise convergence of moment-generating
functions. 
Replacing $s$ by $\frac{s}{\sqrt{\mu_{m(n)}}}$ in \eqref{eq:moment_Cm}, we get
\begin{align*}
 M_{n}(s)
 &:=
 \ETa{\exp\left(\frac{s}{\sqrt{\mu_{m(n)}}} C_{m(n)}\right)}\\
 &=
  \frac{1}{Z_{n,\alpha}}\left[z^{n}\right]
  \exp\left(\vartheta\mathrm{e}^{\frac{s}{\sqrt{\mu_{m(n)}}}}\frac{z^{m(n)}}{m(n)}+\vartheta\sum_{\substack{1\leq j\leq \alpha(n),\\ j\neq m(n)}}\frac{z^{j}}{j}\right).
\end{align*}
In order to determine the asymptotic behaviour of $M_{n}(s)$, we apply Proposition~\ref{prop:SPmethod} with the triangular array $\bsq = (q_{j,n})_{1 \leq j \leq \alpha(n), n \in \N}$ with 
\begin{align}
q_{j,n}
=
\begin{cases}
0 & \text{if }j>\alpha\left(n\right)\\
\vartheta \exp\left(s/\sqrt{\mu_{m(n)}}\right) & \text{if }j=m\left(n\right)\\
\vartheta & \text{otherwise},
\end{cases}
\label{eq:triangular_array_CLT_Cm}
\end{align}
together with $f_n(z)=1$ for all $n$.
    %
    %
    %
    %
    %
    %
We thus have to show that $\bsq$ and the sequence $(f_n)_{n\in\N}$ are both admissible, see Definition~\ref{def:admQ} and~\ref{def:admF}.
The sequence $(f_n)_{n\in\N}$ is admissible for all triangular arrays. 
Thus we have only to show that $\bsq$ is admissible.  
Hence, we have to study the solution $x_{n,\bsq}$ of the equation \eqref{eq:GenSaddle}. 
Since this solution depends on the parameter $s$, we write $x_{n,\bsq}(s)$ instead of $x_{n,\bsq}$.
Also, we will write $\lambda_{2,n}(s)$ for $\lambda_{2,n,\alpha,\bsq}$ with $\lambda_{2,n,\alpha,\bsq}$ as in  \eqref{eq:def_lambda_p}.
We now show
\begin{lem}
\label{lem:HilfeFuerAdm}
Let $\bsq$ be as in \eqref{eq:triangular_array_CLT_Cm} and $x_{n,\bsq}(s)$ be defined as in \eqref{eq:GenSaddle}. 
Suppose that $\mu_{m(n)}\to\infty$ with $\mu_{m(n)}$ as in \eqref{eq:def_mu_n}.
Then we have, locally uniformly in $s\in\R$,
that
\begin{equation}
\alpha\left(n\right)\log\left(x_{n,\bsq}\left(s\right)\right)
\sim
\log\left(\frac{n}{\vartheta\alpha\left(n\right)}\log\left(\frac{n}{\vartheta\alpha\left(n\right)}\right)\right).\label{eq:XnCasymp}
\end{equation}
In particular, if $n$ is large enough,
\[
x_{n,\bsq}(s)\geq1
\text{ and }
\lim_{n\to\infty}x_{n,\bsq}(s)=1.
\]
Furthermore, we have
\begin{equation}
\lambda_{2,n}(s)
\sim 
n \alpha\left(n\right)
\label{eq:Lambda2asymp}
\end{equation}
locally uniformly in $s$ with  $\lambda_{2,n}(s)= \lambda_{2,n,\bsq,\alpha}$ as in \eqref{eq:def_lambda_p}.
\end{lem}
\begin{proof}
We use Lemma~\ref{lem:saddle_point_with_c} to prove Lemma~\ref{lem:HilfeFuerAdm}.
Recall that $x_{n,\alpha}(c)$ is defined in \eqref{eq:def_xn(c)} for $c>0$ as the solution of
\begin{align*}
cn = \vartheta\sum_{j=1}^{\alpha(n)} \big( x_{n,\alpha}(c) \big)^j.
\end{align*}
Furthermore $x_{n,\bsq}(s)$ is the solution of the equation
\begin{align}
 n
 = 
 \vartheta\left(\mathrm{e}^{\frac{s}{\sqrt{\mu_{m(n)}}}}-1\right)\big(x_{n,\bsq}(s)\big)^{m(n)}
 +
 \vartheta\sum_{j=1}^{\alpha(n)}\big(x_{n,\bsq}(s)\big)^{j} .
\label{eq:lem:HilfeFuerAdm1}
\end{align}
We now assume that $0\leq s \leq U$ with $U>0$ an arbitrary, but fixed real number.
Since $\mathrm{e}^{\frac{s}{\sqrt{\mu_{m(n)}}}} \geq 1$, we get
\begin{align}
 x_{n,\bsq}(s) \leq x_{n,\alpha}(1) = x_{n,\alpha},
\end{align}
where $x_{n,\alpha}$ is as in \eqref{eq:StaSad}.
Using the definition of $\mu_{m(n)}$ together with $s\leq U$ and $\mu_{m(n)}\to\infty$, we obtain for $n$ large
\begin{align*}
 \left(\mathrm{e}^{\frac{s}{\sqrt{\mu_{m(n)}}}}-1\right)\big(x_{n,\bsq}(s)\big)^{m(n)}
 &\leq 
 \frac{2U}{\sqrt{\mu_{m(n)}}} \big(x_{n,\alpha}\big)^{m(n)}
 =
 \frac{2U}{\sqrt{\vartheta}} \sqrt{m(n)} \big(x_{n,\alpha}\big)^{\frac{m(n)}{2}}\\
 &\leq
 \frac{2U}{\sqrt{\vartheta}} \sqrt{\alpha(n) \big(x_{n,\alpha}\big)^{\alpha(n)}}.
\end{align*}
Applying Lemma~\ref{lem:saddle_point_with_c} for $x_{n,\alpha} = x_{n,\alpha}(1) $, we get for $n$ large
\begin{alignat*}{1}
\left(\mathrm{e}^{\frac{s}{\sqrt{\mu_{m(n)}}}}-1\right)\big(x_{n,\bsq}(s)\big)^{m(n)}
\leq & 
\frac{4U}{\vartheta}\sqrt{n \log\left(\frac{n}{\vartheta\alpha\left(n\right)}\right)}
\leq n^{1/2+\epsilon},
\end{alignat*}
for $\epsilon>0$ small. 
Inserting this into \eqref{eq:lem:HilfeFuerAdm1}, we get
\begin{align}
\vartheta\sum_{j=1}^{\alpha(n)}\big(x_{n,\bsq}(s)\big)^{j} 
=
n- \vartheta\left(\mathrm{e}^{\frac{s}{\sqrt{\mu_{m(n)}}}}-1\right)\big(x_{n,\bsq}(s)\big)^{m(n)}
\geq 
n(1 - n^{-1/2+\epsilon}).
\end{align}
Using the definition if $x_{n,\alpha}(c)$, we see that 
\begin{align}
x_{n,\alpha}\left(1 - n^{-1/2+\epsilon}  \right) \leq  x_{n,\bsq}(s) \leq x_{n,\alpha}(1).
\end{align}
Applying Lemma~\ref{lem:saddle_point_with_c} to $x_{n,\alpha}\left(1 - n^{-1/2+\epsilon}  \right)$ and $x_{n,\alpha}(1)$ 
immediately completes the proof for $0\leq s\leq U$.
The argumentation for $-U\leq s \leq 0$ is similar and we thus omit it. 
\end{proof}
Lemma~\ref{lem:HilfeFuerAdm} implies that $x_{n,\bsq}(s)$ with $\bsq$ in \eqref{eq:triangular_array_CLT_Cm} is admissible.
Thus we can apply Proposition~\ref{prop:SPmethod}. 
We obtain for each $s\geq 0$ that
\begin{align}
 M_{n}\left(s\right)
 =
 \frac{1}{Z_{n,\alpha}}\frac{\exp\left(h_{n}\left(s\right)\right)}{\sqrt{2\pi\lambda_{2,n}(s)}}\left(1+o\left(1\right)\right),
\end{align}
where
\begin{align}
h_{n}(s)
=
\vartheta\left(\mathrm{e}^{\frac{s}{\sqrt{\mu_{m(n)}}}}-1\right)\frac{\left(x_{n,\bsq}(s)\right)}{m(n)}^{m\left(n\right)}
+
\sum_{j=1}^{\alpha\left(n\right)}\frac{\left(x_{n,\bsq}(s)\right)}{j}^{j}-n\log\left(x_{n,\bsq}\left(s\right)\right).
\label{eq:def_h(s)_CLT}
\end{align}
Since $M_{n}\left(0\right)=1$, we have
\[
\frac{1}{Z_{n,\alpha}}\frac{\exp\left(h_{n}\left(0\right)\right)}{\sqrt{2\pi\lambda_{2,n}(0)}}\xrightarrow{n\to\infty}1.
\]
Our aim is to use this result to complete the proof of Theorem~\ref{thm:Haupt}.
We observe from \eqref{eq:Lambda2asymp} that the leading coefficient of $\lambda_{2,n}(s)$ is independent of $s$.
Therefore, we have proven Theorem~\ref{thm:Haupt} if we can show that for each $s\geq 0$ 
\begin{align}
 h_{n}(s) = h_n(0) + s \sqrt{\mu_{m(n)}}  + \frac{s^2}{2} + o\left(1 \right)
 \qquad
 \text{ as }n\to\infty.
 \label{eq:def_h(s)_CLT_asympt}
\end{align}
We begin with the derivatives of $x_{n,\bsq}(s)$
\begin{lem}
\label{lem:SPabl}
The function $s\mapsto x_{n,\bsq}(s)$ is for each $n$ infinitely often differentiable. Further, we have
\begin{align}
 \frac{x_{n,\bsq}^\prime(s)}{x_{n,\bsq}(s)}
 =
-\frac{\exp\left(\frac{s}{\sqrt{\mu_{m(n)}}}\right)\left(x_{n,\bsq}\left(s\right)\right)^{m\left(n\right)}}{\sqrt{\mu_{m(n)}}\lambda_{2,n}(s)}.
\end{align}
\end{lem}
\begin{proof}
 Let $n$ be fixed. Since all coefficients of $\bsq$ in \eqref{eq:triangular_array_CLT_Cm} are non-negative and not all $0$,
 it follows that the equation \eqref{eq:GenSaddle} has for each $s\geq 0$ exactly one solution. 
 Thus the function $s \to x_{n,\bsq}(s)$ is a well defined function on $[0,\infty)$.
 Applying the implicit function theorem to the function 
 \begin{align*}
  g(s,x) 
  = 
  \vartheta\left(\mathrm{e}^{\frac{s}{\sqrt{\mu_{m(n)}}}}-1\right) x^{m(n)}
 +
 \vartheta\sum_{j=1}^{\alpha(n)}x^{j} 
 \end{align*}
 and using that $\frac{\partial}{\partial x} g(s,x) >0$ for $x>0$ completes the proof.
\end{proof}
Applying Lemma~\ref{lem:SPabl} to $h_n(s)$, we obtain
\begin{lem}
We have
\begin{align}
h'_{n}(s)
&= 
\vartheta\frac{\mathrm{e}^{\frac{s}{\sqrt{\mu_{m(n)}}}}}{\sqrt{\mu_{m(n)}}}\frac{\left(x_{n,\bsq}(s)\right)}{m\left(n\right)}^{m\left(n\right)},
\label{eq:def_h(s)_CLT'}\\
h''_{n}(s)
&= 
\frac{1}{\sqrt{\mu_{m(n)}}} h'_{n}(s)- \frac{(m(n))^2}{\lambda_{2,n}(s)} \left(h'_{n}(s)\right)^2,
\label{eq:def_h(s)_CLT''}\\
h'''_{n}(s)
&=
\frac{1}{\sqrt{\mu_{m(n)}}} h''_{n}(s)
 -
  \frac{2(m(n))^2}{\lambda_{2,n}(s)} h'_{n}(s)h''_{n}(s) 
  + \frac{\lambda_{3,n}(s) x_{n,\bsq}^\prime(s)}{\big(\lambda_{2,n}(s)\big)^2} \left(h'_{n}(s)\right)^2.
  \label{eq:def_h(s)_CLT'''}
\end{align}
\end{lem}
\begin{proof}
 Equation~\eqref{eq:def_h(s)_CLT'} follows immediately from \eqref{eq:def_h(s)_CLT} and the definition of $x_{n,\bsq}(s)$.
 Equation~\eqref{eq:def_h(s)_CLT''} and~\eqref{eq:def_h(s)_CLT'''} follow from Lemma~\ref{lem:SPabl}, equation~\eqref{eq:def_h(s)_CLT'} and 
 the definition of $\lambda_{2,n}(s)$, see Definition~\ref{def:admQ}.
\end{proof}
Equation \eqref{eq:def_h(s)_CLT_asympt} now follows by using 
$x_{n,\bsq}(0) = x_{n,\alpha}$, $\mu_{m} = \vartheta \frac{x^{m}_{n,\alpha}}{m}$
and Lemma~\ref{lem:HilfeFuerAdm}. This completes the proof of Theorem~\ref{thm:Haupt}.
%
%
%
%
%
%
%
%
%
%
%
%
%
%
%
%
%
%
%
%
%
%
%
%

\bibliographystyle{plain}
\bibliography{literatur}

\begin{thebibliography}{10}

\bibitem{ABT02}
R.~Arratia, A.D. Barbour, and S.~Tavar{\'e}.
\newblock {\em Logarithmic combinatorial structures: a probabilistic approach}.
\newblock EMS Monographs in Mathematics. European Mathematical Society (EMS),
  Z{\"u}rich, 2003.

\bibitem{ArTa92c}
Richard Arratia and Simon Tavar{\'e}.
\newblock The cycle structure of random permutations.
\newblock {\em Ann. Probab.}, 20(3):1567--1591, 1992.

\bibitem{BeSc17}
V.~Betz and H.~Sch{\"a}fer.
\newblock The number of cycles in random permutations without long cycles is
  asymptotically {G}aussian.
\newblock {\em ALEA}, 14:427--444, 2017.

\bibitem{BeScZe17}
V.~{Betz}, H.~{Sch{\"a}fer}, and D.~{Zeindler}.
\newblock {Random permutations without macroscopic cycles}.
\newblock December 2017.

\bibitem{BeUe09}
Volker Betz and Daniel Ueltschi.
\newblock Spatial random permutations and infinite cycles.
\newblock {\em Comm. Math. Phys.}, 285(2):469--501, 2009.

\bibitem{BeUe10b}
Volker Betz and Daniel Ueltschi.
\newblock Critical temperature of dilute bose gases.
\newblock {\em Phys. Rev. A}, 81:023611, Feb 2010.

\bibitem{BeUe10}
Volker Betz and Daniel Ueltschi.
\newblock Spatial random permutations and poisson-dirichlet law of cycle
  lengths.
\newblock {\em Electron. J. Probab.}, 16:no. 41, 1173--1192, 2011.

\bibitem{BeUeVe11}
Volker Betz, Daniel Ueltschi, and Yvan Velenik.
\newblock Random permutations with cycle weights.
\newblock {\em Ann. Appl. Probab.}, 21(1):312--331, 2011.

\bibitem{Bi99}
P.~Billingsley.
\newblock {\em Convergence of probability measures}.
\newblock Wiley Series in Probability and Statistics: Probability and
  Statistics. John Wiley \& Sons Inc., New York, second edition, 1999.
\newblock A Wiley-Interscience Publication.

\bibitem{BoZe14}
Leonid~V. Bogachev and Dirk Zeindler.
\newblock Asymptotic statistics of cycles in surrogate-spatial permutations.
\newblock {\em Communications in Mathematical Physics}, pages 1--78, 2014.

\bibitem{DePi85}
J.~M. DeLaurentis and B.~G. Pittel.
\newblock Random permutations and {B}rownian motion.
\newblock {\em Pacific J. Math.}, 119(2):287--301, 1985.

\bibitem{ElPe19}
Dor Elboim and Ron Peled.
\newblock Limit distributions for euclidean random permutations.
\newblock {\em Communications in Mathematical Physics}, 369(2):457--522, 2019.

\bibitem{ErUe11}
Nicholas~M. Ercolani and Daniel Ueltschi.
\newblock Cycle structure of random permutations with cycle weights.
\newblock {\em Random Structures Algorithms}, 44(1):109--133, 2014.

\bibitem{Ew72}
W.~J. Ewens.
\newblock The sampling theory of selectively neutral alleles.
\newblock {\em Theoret. Population Biology}, 3:87--112; erratum, ibid. 3
  (1972), 240; erratum, ibid. 3 (1972), 376, 1972.

\bibitem{FlSe09}
P.~Flajolet and R.~Sedgewick.
\newblock {\em Analytic Combinatorics}.
\newblock Cambridge University Press, New York, NY, USA, 2009.

\bibitem{Ju19}
David {Judkovich}.
\newblock {The Cycle Structure of Permutations Without Long Cycles}.
\newblock May 2019.

\bibitem{Ki77}
J.~F.~C. Kingman.
\newblock The population structure associated with the {E}wens sampling
  formula.
\newblock {\em Theoret. Population Biology}, 11(2):274--283, 1977.

\bibitem{LeTa19}
Benjamin Lees and Lorenzo Taggi.
\newblock Site monotonicity and uniform positivity for interacting random walks
  and the spin o(n) model with arbitrary n.
\newblock {\em Communications in Mathematical Physics}, 2019.

\bibitem{Li10}
Thomas~M. Liggett.
\newblock {\em Continuous time {M}arkov processes}, volume 113 of {\em Graduate
  Studies in Mathematics}.
\newblock American Mathematical Society, Providence, RI, 2010.
\newblock An introduction.

\bibitem{MaPe16}
Eugenijus Manstavi{\v{c}}ius and Robertas Petuchovas.
\newblock Local probabilities for random permutations without long cycles.
\newblock {\em electronic journal of combinatorics}, 23(1), 2016.

\bibitem{Po37}
G.~P{\'o}lya.
\newblock Kombinatorische anzahlbestimmungen f{\"u}r gruppen, graphen, und
  chemische verbindungen.
\newblock {\em Acta Mathematica}, 68:145--254, 1937.

\bibitem{Sc18}
Helge Sch{\"a}fer.
\newblock {\em The cycle structure of random permutations without macroscopic
  cycles}.
\newblock PhD thesis, TU Darmstadt, 2018.

\bibitem{ShVe77}
A.A. Shmidt and A.~M. Vershik.
\newblock Limit measures arising in the asymptotic theory of symmetric groups.
\newblock {\em Theory Probab. Appl.}, 22, No.1:70--85, 1977.

\bibitem{Ta19}
Lorenzo Taggi.
\newblock Uniformly positive correlations in the dimer model and phase
  transition in lattice permutations on $\mathbb{Z}^d$, $d > 2$, via reflection
  positivity.
\newblock 2019.

\bibitem{Ue06}
Daniel Ueltschi.
\newblock Feynman cycles in the {B}ose gas.
\newblock {\em J. Math. Phys.}, 47(12):123303, 15, 2006.

\bibitem{Ya09a}
A.~L. Yakymiv.
\newblock A limit theorem for the middle members of a variational series of
  cycle lengths of random {$A$}-permutation.
\newblock {\em Teor. Veroyatn. Primen.}, 54(1):63--79, 2009.

\bibitem{Ya10a}
A.~L. Yakymiv.
\newblock A limit theorem for the logarithm of the order of a random
  {$A$}-permutation.
\newblock {\em Diskret. Mat.}, 22(1):126--149, 2010.

\bibitem{Ya11}
A.~L. Yakymiv.
\newblock A generalization of the {C}urtiss theorem for moment generating
  functions.
\newblock {\em Mat. Zametki}, 90(6):947--952, 2011.

\end{thebibliography}

\end{document}